\documentclass[10pt]{amsart}
\usepackage{amssymb,amsmath,amsthm,bbm,amsfonts,amsbsy}
\usepackage{stmaryrd}
\usepackage{cancel}
\usepackage[all]{xy}
\usepackage[shortlabels]{enumitem}
\usepackage{ytableau}
\usepackage{color}

\numberwithin{equation}{section}

\newtheorem{thm}{Theorem}[section]

\newtheorem{prop}[thm]{Proposition}
\newtheorem{lem}[thm]{Lemma}
\newtheorem{exer}[thm]{Exercise}

\theoremstyle{definition}
\newtheorem{definition}[thm]{Definition}
\newtheorem{rem}[thm]{Remark}

\theoremstyle{remark}
\newtheorem{example}[thm]{Example}

\allowdisplaybreaks[1]

\begin{document}

\title[Canonical bundles of moving frames]{Canonical bundles of moving frames for parametrized curves in Lagrangian Grassmannians: algebraic approach}

\author
[Igor Zelenko]{Igor Zelenko
	\address{Department of Mathematics, Texas A$\&$M University,
		College Station, TX 77843-3368, USA; E-mail: zelenko@math.tamu.edu}
}
\thanks{ During the preparation of these notes I.\ Zelenko was partly supported by NSF grant DMS-1406193 and Simons Foundation Collaboration Grant for Mathematicians 524213.}

\begin{abstract}
The aim of these notes is to describe how to construct  canonical bundles  of moving frames and differential invariants for parametrized curves in Lagrangian Grassmannians, at least in the monotonic case. Such curves appear as Jacobi curves of sub-Riemannian extremals \cite{agrgam1, jac1} . Originally this construction was done in \cite{li1, li2}, where it uses the specifics of Lagrangian Grassmannian. In later works \cite{flag2, flag1}  a much  more general theory for construction of canonical bundles of moving frames for parametrized or unparametrized  curves in the so-called generalized flag varieties was developed, so that the problem which  is discussed here can be considered as a particular case of this general theory. Although this was briefly discussed at the very end of \cite{flag2}, the application of the theory of \cite{flag2, flag1} to obtain the results of \cite{li1, li2} were never written in detail and this is our goal here. We believe that  this exposition gives a more conceptual point of view on the original results of \cite{li1,li2} and especially clarifies the origin of the normalization conditions of the canonical bundles of moving frames there, which in fact boil down to 
a choice of a complement to a certain subspace of the symplectic Lie algebra. The notes in almost the same form are included under my authorship as the Appendix in  the book ``A Comprehensive Introduction to sub-Riemannian Geometry from Hamiltonian view point'' by A. Agrachev D. Barilari, U. Boscain. Following the style of this book, some not difficult statements are given as exercises for pedagogical purposes.
\end{abstract}

\maketitle

\section{Preliminaries}

\subsection{Basics on moving frames and structure functions}

Throughout these notes $I$ denotes an interval of $\mathbb R$ and the  parameter $t$ takes values in $I$.   
We first start with more elementary and naive point of view on moving frames. For a moment, by a moving frame in $\mathbb R^n$ we mean an $n$-tuple 
$E(t)=\bigl(e_1(t),\ldots e_n(t)\bigr)$ of vectors such that $E(t)$ constitute a basis of $\mathbb R^n$ for every $t$ and it smoothly depends on $t\in I$. We can regard  $E(t)$ as an $n\times n$ matrix with $i$th column equal to the column vector $e_i(t)$, or an element of the Lie group $\mathrm{GL}_n$. So, the moving frame $E(t)$ can be seen as a smooth curve in this Lie group.

The velocity $e_j'(t)$ of the $j$th vector $e_j(t)$ of the moving frame $E(t)$ can be decomposed into the linear combination with respect to the basis $E(t)$, i.e. there exist scalars $r_{ij}(t)$ such that
\begin{equation}
\label{mf1}
e_j'(t)=\sum_{i=1}^n r_{ij}(t) e_i(t), \quad 1\leq j\leq n,\,\ t\in I
\end{equation}
Let $R(t)$ be the $n\times n$ matrix with the $ij$th entry $r_{ij}(t)$. Then \eqref{mf1} is equivalent to 
\begin{equation}
\label{mf2}
E'(t)=E(t) R(t),
\end{equation}
The equation \eqref{mf2} is called the \emph{structure equation of the moving frame $E(t)$} and the matrix function $R(t)$, is called the \emph{structure function of the moving frame $E(t)$}.

\begin{rem}
\label{Maurerrem}
Recall that the (left) Maurer-Cartan form $\Omega$ on a Lie group $G$ with the Lie algebra $\mathfrak g$ is the $\mathfrak g$-valued one-form such that for any $X\in T_a G, a\in G$
\begin{equation}
\label{MC}	
\Omega_a(X):=d(L_a^{-1}) X,
\end{equation}	
where $L_a$ denotes the left translation by an element $a$ in $G$. For a matrix Lie group $\Omega_a(X):=a^{-1}X$, where in the right-hand side the matrix multiplication is used.  Note that \eqref{mf2} can be written as 
\begin{equation}
\label{mf25}
R(t)=E(t)^{-1}E'(t),
\end{equation}
which is equivalent to 
\begin{equation} 
\label{mf3}
R(t)=\Omega_{E(t)}\bigl(E'(t)\bigr),
\end{equation}
where $\Omega$ is the Maurer-Cartan form of $GL_n$, i.e. the structure function of the frame $E(t)$ is equal to the value of the Maurer-Cartan form at the velocity to the frame.
\end{rem}
Now let $G$ be a Lie subgroup of $\mathrm{GL}_n$ with Lie algebra $\mathfrak g$. 
We will say that a moving frame $E(t)$ is \emph{$G$-valued} if $E(t)$, considered as an $n\times n$- matrix, belongs to $G$.   
From \eqref{mf3} it follows that the structure function of a $G$-valued frame takes value in the Lie algebra $\mathfrak g$.

Two $G$-valued moving frames $E(t)$ and $\widetilde E(t)$ are called equivalent  with respect to $G$, or \emph{$G$-equivalent}, if there exists $A\in G$ such that $\tilde e_j(t)= A e_j(t)$ for any $1\leq j\leq n$ and $t\in I$ or, equivalently, in the matrix form
\begin{equation}
\label{equiv1}
\widetilde E(t)=A E(t), \quad \forall t\in I
\end{equation} 
The following simple lemma is fundamental for the applications of moving frames in geometry of curves:

\begin{lem}
\label{fundlem}
\begin{enumerate}
\item Two $G$-valued moving frames $E(t)$ and $\widetilde E(t)$ with the structure functions $R(t)$ and $\widetilde R(t)$, respectively  are $G$-equivalent if and only if $R(t)\equiv \widetilde R(t)$ on $I$. 

\item Given any function $R:I\rightarrow \mathfrak g$ there exists the unique, up to the action of $G$,  $G$-valued moving frame with the structure function  $R(t)$.
\end{enumerate}	
\end{lem}

\begin{proof} The ``only if'' part of the first statement of the lemma is trivial, because
	if $\widetilde E(t)=A E(t)$, then 
	$$\widetilde R(t)=\bigl(\widetilde E(t)\bigr)^{-1} \widetilde E'(t)=\bigl(E(t)^{-1} A^{-1}\bigr) \bigl(A E'(t)\bigr)=E(t)^{-1} E'(t)=R(t).$$
	
	For the ``if'' part take $t_0\in I$ and let $A:=\widetilde E(t_0) E(t_0)^{-1}$. Then clearly  
	\begin{equation}
	\label{initial}
	\widetilde E(t_0) =A E(t_0).
	\end{equation}
 Further, by the same arguments as in the previous part, the moving frame $AE(t)$ has the same structure function as $E(t)$ and therefore, by our assumptions, as $\widetilde E(t)$. In other words,  the frames $AE(T)$ and $\widetilde E(t)$ satisfy the same system of linear ODEs. Since by \eqref{initial} they meet the same initial conditions at $t_0$, we have \eqref{equiv1} by the uniqueness theorem for linear ODEs. 
 
The existence claim of the second statement of the lemma follows from the existence theorem for linear systems of ODEs, while the uniqueness part follows from the ``if'' part of the first statement. 
\end{proof}

\subsection{Applications to geometry of curves in Euclidean space}
The previous lemma is the basis for application of moving frames and construction of the complete system of invariants for various types of curves  with respect to the action of various groups. Perhaps, every Differential Geometry student quickly  encounters the  Frenet-Serret moving frame in the study of curves in Euclidean space up to a rigid motion\footnote{Here by the rigid motion we mean the map $x\mapsto a+Ux$, where $U\in O_n$, while often one assumes that $U\in SO_n$}. Recall its construction: Assume that a curve $\gamma(t)$ in $\mathbb R^n$ is parametrized by an arc length and for simplicity \footnote{Often, especially in the case of a modification of the problem mentioned in the previous footnote,one makes a weaker assumption that $\dim \,\,\mathrm {span} \{\gamma'(t),\ldots \gamma^{(n-1)}(t)\}=n-1$}  satisfies the following regularity assumption:

\begin{equation}
\label{reg}
\mathrm {span} \{\gamma'(t),\ldots \gamma^{(n)}(t)\}=R^n
\end{equation}
i.e  $\Bigl(\gamma'(t),\ldots \gamma^{(n)}(t)\Bigr)$ is a moving frame. The Frenet-Serret moving frame is obtained from this frame by the Gram-Schmidt orthogonalization procedure. This is $O_n$-valued (or orthonormal) moving frame\footnote{In the case of the weaker assumption  of the previous footnote related to the problem mentioned in the first footnote one uses the Gram-Schmidt orthogonalization for the tuple of vectors $\{\gamma'(t),\ldots \gamma^{(n-1)}(t)\}$ to construct $n-1$ unit and pairwise  orthogonal vectors and then completes this to $\mathrm{SO}_n$-valued frame.} uniquely assigned to each curve with the properties above and two curves can be transformed one to another by a rigid motion if and only of their Frenet-Serret frames are $O_n$-equivalent, which by Lemma \ref{fundlem} is equivalent to the fact that the structure functions  of their Frenet-Serret frames coincide. By constructions, these structure functions are $\mathfrak {so}_n$-valued (i.e. skew-symmetric) such that  the following condition holds:
\medskip

{\bf \emph{The normalization condition for the Frenet-Serret frame:}} \emph{ The only possible nonzero entries below the diagonal are $(j+1, j)$-entry with $1\leq j\leq n-1$.}
\medskip
These $n-1$ entries completely determine the structure function by skew-symmetricity and therefore, again by Lemma \ref{fundlem}, constitute  constitute its complete system of invariants. The $(2,1)$-entry classically called the curvature of the curve, the $(3,2)$-entry is classically called the torsion, at least for $n=3$, and for higher dimensions all other non-zero entries are called higher  order curvatures.

Already in this classical example one encounters, at least implicitly, the notion of the curve of osculating flags, which will be important in the sequel. 
Recall that a \emph{flag} in a vector space $V$ or a \emph{filtration} of $V$ is a collection of nested subspaces of $V$. The flag is 
called \emph{complete}, if each dimension between $0$ and $\dim V$ appears exactly once in the collection of dimensions of the subspaces of the flag.  To a curve $\gamma$,  satisfying  regularity assumption \eqref{reg}, one can assign the following curve of complete flags in 
$\mathbb R^n$:
\begin{equation}
\label{oscflag}
0\subset\mathrm{span}\{\gamma'(t)\}\subset \mathrm{span}\{\gamma'(t),\gamma''(t)\}\subset\ldots \subset \mathrm{span}\{\gamma^{(j)}(t)\}_{j=1}^{i}\subset\ldots\subset  \mathrm{span}\{\gamma^{(j)}(t)\}_{j=1}^{n}=\mathbb R^n,
\end{equation}
which is called the \emph{curve of osculating flags}, associated with the curve $\gamma$. A moving frame $(v_1(t),\ldots, v_n(t)$ is called \emph{adapted} to the curve of osculating flags \eqref{oscflag}, if $\mathrm{span}\{v
	_j(t)\}_{j=1}^{i}=\mathrm{span}\{\gamma^{(j)}(t)\}_{j=1}^{i}$ for all $1\leq i\leq n$. The Gram-Schmidt orthogonalization procedure is nothing but the procedure of construction of the orthonormal frame adapted to the curve of flag \eqref{oscflag} and such that the $i$th vector of the frame points toward the same half-space of the $i$th subspaces of the flag \eqref{oscflag} with respect to the $(i-1)$th subspace of this flag as the vector $\gamma^{(i)}(t)$. In this case there is exactly one such adapted frame.
	
	In this example, one  can see another important point for our exposition. The Frenet-Serret frame can be described in terms of its structure function without referring to the Gram-Schmidt orthogonalization: The Frenet-Serret frame  of a curve $\gamma$ parametrized by an arc length  is the only orthonormal frame such that the very first vector is $\gamma'(t)$ and the structure function satisfies the \emph{normalization conditions} above with all nonzero entries being positive.

\subsection{More general point of view: homogeneous spaces and moving frames as lifts to the group, defining the equivalence}
\label{liftsec}
Equivalence problem for curves in $\mathbb R^n$ up to a rigid motion is a particular case of equivalence problem for curves in a homogeneous space. In more detail, given a Lie group $G$ and its closed subgroup $G^0$, the space $G/G^0$ of the left cosets of $G^0$ is a smooth manifold with the natural transitive action of $G$ induced by the left translation on $G$. The space $G/G^0$ is called a \emph{homogeneous space} of the group $G$.  Two curves in $G/G^0$ are called equivalent if there exist and element of $G$ sending one curve to another. In the case of curves in $\mathbb R^n$ up to a rigid motion, $G$ is the group of rigid motion, denoted by $\mathrm{AO}_n$ and $G^0$ can be taken as  its subgroup preserving the origin of $\mathbb R^n$, i.e. the group of orthogonal transformations $O_n$, so that $\mathbb R^n\cong \mathrm{AO}_n/O_n$. 

\emph{Grassmannians} and, more generally, \emph{flag varieties} provide another class of examples of homogeneous spaces. A flag variety is a set of flags of a vector space $V$ with fixed  dimensions of subspaces in these flags. Fix one of the flags and let $G^0$ be the subgroup of $\mathrm{GL}(V)$ preserving this flag. Then the flag variety can be identified with $GL(V)/G^0$.
A Grassmannian corresponds to a flag variety with flags consisting of a one subspace of a fixed dimension. 

If $V$ has some additional structure, then one can distinguish special flags and consider proper subgroups of $GL(V)$ as the group $G$. For example, let  $V$ be a $2m$-dimensional vector space endowed with a symplectic, i.e. a non-degenerate skew-symmetric, form $\sigma$. Given a subspace $\Lambda$ of $V$ let $\Lambda^\angle$ be the \emph{skew-orthogonal complement} of $\Lambda$ with respect to $\sigma$,
$$\Lambda^\angle=\{v\in V: \sigma(v, z)=0 \,\,\forall z\in \Lambda\} $$
Then one can distinguish the following classes of subspaces of $V$: a subspace $\Lambda$ of $V$  is called \emph{isotropic}, if $\Lambda \subset \Lambda^\angle$ or, equivalently. $\sigma|_{\Lambda}=0$, a subspace $\Lambda$ is called \emph{coisotropic}, if $\Lambda^\angle\subset\Lambda$, or equivalently, if $\Lambda^\angle$ is isotropic, and it is called  \emph{Lagrangian} if it is both isotropic and coisotropic, i.e. $\Lambda=\Lambda^\angle$. Since from nodegenericty of $\sigma$ subspaces $\Lambda$ and $\Lambda^\angle$ have complementary dimensions to $2m$, the dimension of isotropic spaces is not greater than $m$ and of the coisotropic subspaces is not smaller than $m$. Hence  the dimension of Lagrangian subspaces is equal to $m$. The set $L(V)$ of all Lagrangian subspaces is called the \emph{Lagrangian Grassmannian}. Let $\mathrm{Sp}(V)$ be the group of all symplectic transformations, i.e. of all $A\in \mathrm{GL}(V)$ preserving the form $\sigma$, $\sigma(Av, Aw)=\sigma(v,w)$. Since $\mathrm{Sp}(V)$ acts transitively on $L(W)$, the latter can be identified with $\mathrm{SP}(V)/G^0$, where $G^0$ is the subgroup of $\mathrm{Sp}(V)$ preserving one chosen Lagrangian subspace $\Lambda$.

Initially we defined a moving frame as a one-parametric family  of bases in $\mathbb R^n$. However the most important thing in the equivalence problem was the structure function of the frame and in order to define it we  essentially used the corresponding one-parametric family (i.e. a curve) of matrices. This motivates the following slightly more abstract definition of a moving frame:  a \emph{ moving frame in a vector space $W$} is a curve  $E(t)$ of bijective linear operators on $W$, i.e. a curve in $GL(W)$. In the same way if $G$ is a Lie subgroup of $GL(W)$, then the $G$-valued moving frame is a curve in $G$. The structure function of $E(t)$ is defined by equation \eqref{mf25}.

In order to relate this point of view to the previously defined notion of a moving frames in $\mathbb R^n$  it is just enough to choose some basis in $W$. This identifies $W$ with $\mathbb R^n$ and the operators of the frame with the matrices with respect to this basis. The one-parametric family of bases in $W$ can be obtained by taking the images of the chosen basis under the operators of the moving frame. 

Since the structure functions of the moving frame are defined via the Maurer-Cartan, which is defined on any Lie group,  we can go further and give the following definition of a moving frame in a homogeneous space of an abstract Lie group without any relation to the initial naive notion of the moving frame as a curve of bases: 
\begin{definition}
A moving frame  over a curve $\gamma$ in a homogeneous space $G/G^0$ is  a \emph{smooth lift} of $\gamma$ to the Lie group $G$, i.e. a smooth curve $\Gamma$ in $G$ such  that  $\pi\bigl(\Gamma(t)\bigr)=\gamma(t)$ for every $t$, where $\pi:G\rightarrow G/G^0$ is the canonical projection. The structure function of the moving frame $\Gamma(t)$ is the $\mathfrak g$-valued function
\begin{equation}
\label{structdef}
C_\Gamma(t):=\Omega_{\Gamma(t)}(\Gamma'(t))
\end{equation}
where $\Omega$ is the left Maurer -Cartan form on the Lie group $G$.
\end{definition}
 This definition can be related to the previous one if one choose a faithful representation of $G$.
 
 Note that in the case of curves in a flag variety $G/G^0$  the set of  moving frames over this curve can be naturally  related  to the set of moving frames adapted to this curve of flags  in the sense of the previous subsection.

For completeness, let us adjust this more general point of view on moving frames for curves in homogeneous spaces to  the  case of curves in Euclidean space.
Here a standard representation of the affine group, although one can easily manage without any representation.
For this identify $\mathbb R^n$ with the affine subspace of $\mathbb R^{n+1}$ by identifying a point $(x_1, \ldots, x_n)\in \mathbb  R^n$ with the point $(1, x_1, \ldots x_n)\in\mathbb R^{n+1}$ and the group of rigid motions $\mathrm{AO}_n$  with the subgroup of $\mathrm{GL}_{n+1}$, consisting of the matrices of the form 

\begin{equation}
\label{affine}
\left(\begin{array}{c|c}
1&0\\
\hline\\
a& U
\end{array}
\right),
\end{equation}
where $a\in \mathbb R^n$ and $U\in \mathrm{O}_n$. Here the matrix \eqref{affine} corresponds to the rigid motion sending $x\in \mathbb R^n$ to $a+Ux$,  because this matrix sends  the  vector $\begin{pmatrix} 1\\ x\end{pmatrix}$, $x\in \mathbb R^n$  to $\begin{pmatrix}1\\ a+Ux \end{pmatrix}$. To a curve $\gamma$ we assign the moving frame, given by the following curve $\Gamma(t)$ in $\mathrm{AO}_n\subset \mathrm {GL}_{n+1}$   
\begin{equation}
\label{frameaff}
\Gamma(t)=\begin{pmatrix} 1&0&\ldots& 0\\ \gamma(t)& e_1(t)&\ldots&e_n(t)\end{pmatrix},
\end{equation}
where $\bigl(e_1(t),\ldots, e_n(t)\bigr)$ is the Frenet-Serret frame of $\gamma(t)$. Since by constructions $\gamma'(t)=e_1(t)$, the structure function of this frame is 
$$\left(\begin{array}{c|c}
0&\begin{array}{ccc}0&\ldots& 0\end{array}\\
\hline\\
\begin{array}{c}1\\0\\ \vdots\\0\end{array}& R(t)
\end{array}
\right),$$
where $R(t)$ is the structure function of the Frenet-Serret frame of $\gamma$.
In contrast to the Frenet-Serret frame, the frame \eqref{frameaff} takes values in  the entire  group of rigid motions that defines the considered equivalence relation. Hence, equivalence of curves with the same structure functions of their $\mathrm{AO}_n$-valued frames, as given in \eqref{frameaff}, is obtained immediately  from Lemma \ref{fundlem}, while using the Frenet-Serret frame, Lemma \ref{fundlem} gives that the curve of velocities  are equivalent by an orthogonal transformation and the equivalence of the original curves up to a rigid motion is obtained after integration only. Although  technically these two arguments are equally elementary, the use of frames, which  take values in the entire group defining the equivalence relation (or, equivalently, in the entire group of the homogeneous space) has an obvious conceptual advantage.

The moving frame  $\Gamma(t)$ from \eqref{frameaff} can be seen as the  
\emph{canonical lift} of the original curve $\gamma$ from $\mathbb R^n=\mathrm{AO}_n/O_n$   to the group $\mathrm{AO}_n $. In  view of Lemma \ref{fundlem}, the construction of such a canonical lift or a bundle of such lifts is the main idea for solving such type of equivalence problems. Canonical means that two curves $\gamma(t)$ and $\widetilde\gamma(t)$ in $G/G^0$ are equivalent via $g\in G$, i.e. $g.\gamma(t)=\widetilde \gamma(t)$ if and only if $g$ sends any canonical lift of $\gamma(t)$ to a canonical lift of $\tilde\gamma(t)$.  
	

\subsection{Some general ideas on canonical bundles of moving frames: symmetries and normalization conditions}
A way to choose  the canonical moving frame or, equivalently, the canonical lift to the group $G$, is to specify  certain restrictions, called \emph{normalization conditions}, on its structure function 
In the case of curves in  Euclidean space the necessity of such specification does not really emerge, because it follows automatically from the condition on the moving frame to be adapted to the osculating flag.  Besides, in this case such a frame is unique. 

For curves in  general homogeneous spaces $G/G^0$  one cannot expect that there exists a unique canonical lift to the corresponding Lie group. The reason for this is that a curve $\gamma(t)$ in $G/G^0$ may have a nontrivial \emph{non-effective} symmetry, i.e. an element $s$ of $G$, which is not the identity, but  preserves each point of $\gamma$,  $s.\gamma(t)=\gamma(t)$ for every $t$. The group of non-effective  symmetries will be denoted by $\mathrm{Sym}^{\mathrm{ne}}_\gamma$.  This group of symmetries  is relevant to geometry of parametrized  curves. In the case of unparametrized curves they should be replaced by a larger subgroup of $G$ consisting of symmetries preserving a distinguished point of $\gamma$ only and not necessary other points of $\gamma$.



If $s\in \mathrm{Sym}^{\mathrm{ne}}_\gamma$ and $\Gamma(t)$ is a lift of $\gamma(t)$  then $s \Gamma(t)$ is a lift of $\gamma(t)$ as well and there is not any preference of choosing $\Gamma(t)$ over  $s \Gamma(t)$ and vice versa. In particular, they have the same structure functions. In other words the group $\mathrm{Sym}^{\mathrm{ne}}_\gamma$ encodes the minimal possible freedom of choice of a canonical frame for the curve $\gamma(t)$, so that if a normalization condition on the structure function is chosen, the set of all lifts satisfying this condition forms a fiber bundle over the curve $\gamma$ with the fibers of dimension not smaller than $\dim \mathrm{Sym}^{\mathrm{ne}}_\gamma$ and which is foliated by the lifts. This bundle will be referred as the \emph{canonical bundle} corresponding to the chosen normalization conditions.

 The reason why we do not insist that such bundle will have fibers of dimension exactly equal to $\dim \mathrm{Sym^{ne}}_\gamma$ is that  different curves in $G/G^0$ may have non-effective symmetry groups of different dimension. For example,  generic curves may not have any nontrivial non-effective symmetry. So, if we insist to assign to each curve the bundle of moving frames of the minimal possible dimension, we may have too much branching in this construction with different normalization conditions for each branch. Instead, it is preferable to choose the widest possible classes of curves so that within each class the uniform normalization conditions for moving frames are used and so that for some distinguished curves in this class the symmetry group has the maximal possible dimension within the class , i.e. for these curves the canonical bundle is of the smallest possible dimension.
 The natural candidates for such curves are orbits of one- parametric subgroups of $G$ and in the case of Grassmannians and flag varieties these one-parametric subgroups are generated by nilpotent elements of the Lie algebra of $G$.

 \section{Algebraic theory of curves in Grassmannians and flag varieties} 
 
  It turns out that for curves in Grassmannians and, more generally,  in flag varieties all main steps in the construction of canonical moving frames, including the choice of the class of curves, the description of maximal group of symmetries and of the normalization conditions can be maid purely algebraically.  In this section we we will describe this algebraic theory.
 
 \subsection{Tangent spaces to  Grassmannians}
 \label{tansec}
 Here we discuss this topic from the point of view of homogeneous spaces.   
 The tangent space to a Lie group $G$ at a point $a$ can be identified with its  Lie algebra $\mathfrak{g}$ via the Maurer-Cartan form as in \eqref{MC}. This immediately implies that the tangent space to a homogeneous space $G/G^0$ at a point $o$ can be identified with the quotient space $\mathfrak{g}/\mathfrak{g}^0$, 
 \begin{equation}
 \label{tanhom}
 T_o(G/G^0)\cong \mathfrak g/\mathfrak g^0
 \end{equation}
 where $\mathfrak g$ and $\mathfrak g^0$ are the Lie algebras of $G$ and $G^0$, respectively. 
 
 Consider the case of the Grassmannian $\mathrm {Gr}_k(V)$ of $k$-dimensional subspaces in a vector space $V$. Fix a point $\Lambda\in \mathrm {Gr}_k(V)$. 
 As already mentioned in subsection \ref{liftsec}, $\mathrm {Gr}_k(V)$ can be identified with $GL(V)/G^0$, where $G^0$ is the subgroup of $GL(V)$, preserving the subspace $\Lambda$.
The Lie algebra $\mathfrak{gl}(V)$  of $\mathrm{GL}(V)$ is the algebra  
of all endomorphisms of $V$ and the Lie algebra $\mathfrak g^0$ of $G^0$ is the algebra of all endomorphisms of $V$, preserving $\Lambda$. It is easy to see that the quotient space $\mathfrak{gl} (V)/\mathfrak g^0$ can be canonically identified with the space $\mathrm{Hom}(\Lambda, V/\Lambda)$. Indeed, let  $p:V\rightarrow V/\Lambda$ be the canonical projection to the quotient space.  The assignment
\begin{equation}
\label{quotientid}
 A\in  \mathfrak{gl} (V)\mapsto (p\circ A)|_\Lambda\in \mathrm{Hom}(\Lambda, V/\Lambda)
\end{equation} 	
maps the endomorphisms from the same coset in   $\mathfrak{gl} (V)/\mathfrak g^0$  to the same element of   $\mathrm{Hom}(\Lambda, V/\Lambda)$ and is onto.		
Therefore  by \eqref{tanhom}
 \begin{equation}		
 \label{tangrass}
 T_\Lambda \mathrm{Gr}_k(V)\cong \mathrm{Hom}(\Lambda, V/\Lambda).
\end{equation} 
This identification can be described also as follows:
Take  $Y\in T_\Lambda \mathrm{Gr}_k(V)$ and let $\Lambda(t)$ be a curve in $\mathrm{Gr}_k(V)$ such that $\Lambda(0)=\Lambda$ and $\Lambda'(0)=Y$. 	  
 Given $l\in \Lambda$ take a smooth curve of vectors $\ell(t)$ satisfying the following two properties:
 \begin{enumerate}
 	\item $\ell(0)=l$,
 	\item $\ell(t)\in\Lambda(t)$ for every  $t$ close to $0$.
 \end{enumerate}
 
 \begin{exer} Show that the coset of $\ell'(t)$
 in $V/\Lambda$ is independent of the choice of the curve $\ell$ satisfying the properties (1) and (2) above. 
\end{exer}

Based on the previous exercise to $Y$ we can assign the element of
 ${\rm
 	Hom}\bigl(\Lambda(t), W/\Lambda(t)\bigr)$ that sends $l\in \Lambda$ to the coset of $\ell'(0)$
 in $V/\Lambda$, where the curve $\ell$ satisfies properties (1) and (2) above.
 
 Now assume that $V$ is $2m$-dimensional and is endowed with a symplectic form $\sigma$. Describe the identification analogous to \eqref{tangrass} for the Lagrange Grassmannian $L(V)$. Take $\Lambda\in L(V)$. Since $L(V)\subset \mathrm{Gr}_m(V)$, the space $T_\Lambda L(V)$ can be  identified with a subspace of $\mathrm{Hom}(\Lambda, V/\Lambda)$.
 To describe this subspace, first note that $\sigma$ defines the identification of $V/\Lambda$ with the dual space $\Lambda^*$: the assignment
\begin{equation}
\label{dualid}
v\in V\mapsto (i_v\sigma)|_\Lambda 
\end{equation}
maps the elements from the same coset of $V/\Lambda$ to the same element of $\Lambda^*$ and is onto. Hence, it defines the required identification. Here $i_v\sigma$ defines the interior product of the vector $v$ and the form $\sigma$, that is  $i_v\sigma (w)= \sigma(v, w)$.
Second, since $L(W)\cong \mathrm{Sp}(V)/G^0$, where $G^0$ is the subgroup of $\mathrm {Sp}(V)$ preserving the space $\Lambda$,  by \eqref{tanhom} the space $T_\Lambda L(V)$ can be identified with 
$\mathfrak{sp}(V)/\mathfrak{h}$, where $\mathfrak{sp}(V)$ is the Lie algebra of $\mathrm{Sp}(V)$ and it consists of $A\in \mathfrak{gl}(V)$ such that
\begin{equation}
\label{sympalg}
\sigma(Av, w)=\sigma(Aw, v), 
\end{equation}
i.e. the bilinear form $\sigma(A\cdot, \cdot)$ is symmetric,		 
and $\mathfrak g^0$ is the Lie algebra of $G^0$.  So, in \eqref{quotientid} we have to take $A\in \mathfrak{sp}(V)$. From  \eqref{dualid} and \eqref{sympalg} it follows that $(p\circ A)|_\Lambda$ considered as a map from $\Lambda$ to $\Lambda^*$ is self-adjoint. Besides, any self-adjoint map from $\Lambda$ to $\Lambda^*$ can be obtained in this way from some $A\in \mathfrak{sp}(V)$. The space of self-adjoint maps from $\Lambda$ to $\Lambda^*$  can be identified with the space $\mathrm{Quad}(\Lambda)$ on $\Lambda$,
\begin{equation}		
\label{tanlg}
T_\Lambda L(V)\cong \mathrm{Quad}(\Lambda).
\end{equation}
Similarly to the case of the Grassmannian, if   $Y\in T_\Lambda L(V)$, $\Lambda(t)$ is  a curve in $L(V)$ such that $\Lambda(0)=\Lambda$ and $\Lambda'(0)=Y$, and  a smooth curve of vectors $\ell(t)$ satisfies the properties (1) and (2) above, then the quadratic form on $\Lambda$ corresponding to $Y$ is the form sending $l$ to $\sigma(\ell'(0), l)$.

We say that a curve $\Lambda(t)$ in Lagrange Grassmannians is \emph{monotonically nondecreasing} if it velocity $\cfrac{d}{dt} \Lambda(t)$ is non-negative quadratic form for every $t$.
   
\subsection{Osculating flags and symbols of  curves in Grassmannians}
\label{oscflagsec}
The goal of this subsection is to distinguish the classes of curves in Grassmannians for which the uniform construction of canonical moving frames can be made.   
For this, first  to a curve $\Lambda(t)$ in the Grassmannian $\mathrm{Gr}_k(V)$ we assign a special curve of flags, called the \emph{curve of osculating flags}. Denote by $C(\Lambda)$ the canonical bundle over the curve $\Lambda$: The fiber of $C(\Lambda)$ over the point $\Lambda(t)$ is the vector space $\Lambda(t)$. Let $\Gamma(\Lambda)$ be the space of all sections of $C(\Lambda)$ . Set
$\Lambda^{(0)}(t):=\Lambda(t)$ and  define inductively
$$\Lambda^{(-j)}(t)={\rm span}\,\left\{\frac{d^k}{d t^k}\ell(t): \ell\in\Gamma(\Lambda), 0\leq k\leq j\right\}$$ for $j\geq0$.
The space $\Lambda^{(-j)}(t)$, $j>0$, is called the \emph{$j$th extension} or the \emph{$j$th osculating subspace} of the  curve $\Lambda$ at point $t$. The usage of  negative indices here is in fact natural because, as we will see later, it is in accordance with the order of invariants of the curve  (i.e. the order of jet of a curve on which an invariant depends) and also with the  natural filtration of the algebra of infinitesimal symmetries, given by stabilizers of jets of a curve of each order (and indexed by this jet order).

Further, given a subspace $L$ in $V$ 
denote by $L^\perp$ the annihilator of $L$ in the dual space $V^*$:
$$L^\perp=\{p\in V^*:p(v)=0,\,\forall \,v\in L\}.$$ 
Set
\begin{equation}
\label{contra}
\Lambda^{(j)}(t)=\Bigl(\bigl(\Lambda(t)^\perp\bigr)^{(-j)}\Bigr)^\perp, \quad j\geq 0.
\end{equation}
The subspace $\Lambda^{(j)}(t)$, $j>0$, is called the \emph{$j$th contraction} of the  curve $\Lambda$ at point $t$.
Clearly, $\Lambda^{(j)}(t)\subseteq\Lambda^{(j-1)}(t)$. The flag (the filtration)
\begin{equation}
\label{filt}
\ldots\Lambda^{(2)}(t)\subseteq\Lambda^{(1)}(t)\subseteq\Lambda^{(0)}\subseteq \Lambda^{(-1)}(t)\subseteq \Lambda^{(-2)}(t)\subseteq\ldots
\end{equation} 
is called the \emph{osculating flag (filtration) of the curve $\Lambda$ at point $t$.} By construction, the curves in Grassmannians (Lagrangian Grassmannian) are $\mathrm{GL}(V)$-equivalent ($\mathrm {Sp}(V)$-equivalent) if and only if the curves of their osculating flags  are $\mathrm{GL}(V)$-equivalent ($\mathrm {Sp}(V)$ equivalent). 

 A flag $\{X^j\}_{j\in\mathbb Z}$ in a symplectic space $X$ with $X^j\subset X^{j-1}$ will be called \emph{symplectic} if all subspaces $X^j$ with $j>0$  are coisotropic and $X^{-j}=(X^j)^\angle$. Then all subspaces $X^j$ with $j<0$ are isotropic. If $\Lambda(t)$ is a curve in a Lagrangian Grassmannian, then the flag \eqref{filt} is symplectic. Indeed, as $\Lambda(t)\subset \Lambda^{(j)}(t)$ for $j<0$ and $\Lambda(t)$ is Lagrangian, we have that $$\Lambda^{(-j)}=\Lambda^{(j)}(t)^\angle\subset \Lambda(t)\subset \Lambda^{(j)}(t),$$ which implies that the subspaces $\Lambda(t)\subset \Lambda^{(j)}(t)$ are isotropic.
Further, from \eqref{dualid} and \eqref{contra} it follows that   
	\begin{equation}
	\label{skewsymm}
	\Lambda^{(j)}(t)=\Bigl(\Lambda^{(-j)}(t)\Bigr)^\angle.
	\end{equation}

The curve $\Lambda(t)$ is called \emph{equiregular} if for every $j>0$ the dimension of $\Lambda^{(j)}(t)$ is constant.  As the integer-valued function $\dim \Lambda^{(j)}(t)$ is lower semi-continuous, it is locally constant on an open dense set of $I$, which will imply that for a generic $t$ the curve $\Lambda$ is equiregular in a neighborhood of $t$. So, from now on we will assume that the curve $\Lambda(t)$ is equiregular. For an equiregular curve  passing to the osculating flag, we get a curve in a flag variety, i.e. the equivalence of curves in Grassmannians (Lagrangian Grassmannians) is reduced to the equivalence of the osculating curves in the corresponding flag varieties.


\begin{rem}
	For an equiregular curve $\Lambda(t)$ the subspaces $\Lambda^{(j)}(t)$ can be described by means of the identification of tangent vectors to Grassmannians with certain linear maps as in subsection \ref{tansec}. Namely, $\Lambda^{(j-1)}(t)$ is the preimage under the canonical projection from $V$ to $V/\Lambda^{(j)}(t)$ of $\mathrm {Im} \frac{d}{dt} \Lambda^{(j)}(t)$. For an equiregular curve $\Lambda(t)$ we also have 
	$$\Lambda ^{(j)}(t)=\mathrm {Ker}\frac{d}{dt}\Lambda ^{(j-1)}(t), \quad j>0$$
\end{rem}

\begin{exer}
Prove that if the curve $\Lambda$ is equiregular, then  
$$\dim\Lambda^{(j-1)}(t)-\dim \Lambda^{(j)}(t)\leq \dim
\Lambda^{(j)}(t)-\dim \Lambda^{(j+1)}(t), j<0.$$
\end{exer} 
Recall that a Young diagram is  a finite collection of boxes,  arranged in columns (equivalently, rows) with the column (rows) lengths in non-increasing order , aligned  from the top and the left, as shown in the example below:

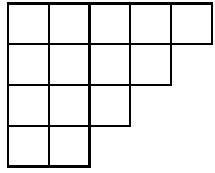
\begin{equation*}
\label{youngodd}
\begin{ytableau}
~&~&~ & ~&~\\
~&~&~ &\\
~&~&\\
~&
\end{ytableau} 
\end{equation*}
\begin{definition}
\label{Youngdef}
The Young diagram $D$ such that  the number of boxes in the
$-j$th column, $j<0$, of $D$ is equal to $\dim \Lambda^{(j)}-\dim
\Lambda^{(j+1)}$ is called  
the Young diagram $D$ of the curve $\Lambda(t)$ in Grassmannian
\end{definition}	
This notion is especially important for monotonic curves in Lagrangian Grassmannians as shown in Proposition \ref{Youngprop}.

\begin{example}
\label{regexample}	
	 The curve in Lagrangian Grassmannian $\Lambda$ is called regular if $\Lambda^{(-1)}=V$. In this case the Young diagram consists of one row with the number of boxes equal to $\cfrac{1}{2}\dim V$. $\Box$.
\end{example} 
Let 
\begin{equation}
\label{gradj}
V_j(t):=\Lambda^{(j)}(t)/\Lambda^{(j+1)}(t)
\end{equation}
and 
\begin{equation}
\label{grad}
\mathrm {gr}V(t):=\bigoplus_{j\in\mathbb Z}V_j(t)
\end{equation}
be the graded space, associated with the filtration \eqref{filt}.
By constructions, for any $j\in\mathbb Z$ we have the following inclusion 
\begin{equation}
\label{compatosc}
(\Lambda^{(j)})^{(-1)}(t)\subseteq\Lambda^{(j-1)}(t).
\end{equation}

Hence, the velocity $\frac{d}{dt}\Lambda ^{(j)}(t)$ of the curve $\Lambda^{(j)}$ at $t$, which is the map from $\Lambda^{(j)}(t)$ to $V/\Lambda^{(j)}(t)$ factors through the map from $V_j(t)$ to $V_{j-1}(t)$. Thus, the velocity of the curve of osculating flags \eqref{filt} at $t$ factors through the endomorphism $\delta_t$ of the graded space $\mathrm {gr} V(t)$, sending  $V_j(t)$ to $V_{j-1}(t)$ for any $j\in \mathbb Z$, i.e. the degree $-1$ endomorphism of the graded space $\mathrm {gr}V(t)$. This endomorphism is called the \emph{symbol} of the curve $\Lambda$ at $t$ in the Grassmannian $Gr_k(V)$.

\begin{rem}
\label{surinrem}	
By constructions $\delta_t:V_j(t)\rightarrow V_{j-1}(t)$ is injective for $j\geq 0$ and surjective for $j\leq 0$.
\end{rem}
 
The natural equivalence relation on the space of endomorphisms of graded vector spaces is defined via conjugation: two endomorphisms $\delta$ and $\tilde\delta$ acting on graded vector spaces $X=\displaystyle{\bigoplus_{j\in \mathbb Z}} X_j$ and $\widetilde X=\displaystyle{\bigoplus_{j\in \mathbb Z}} \tilde X_j$, respectively, are called equivalent, if there exists an isomorphism $Q:X\rightarrow \tilde X$, preserving the grading, i.e. such that $Q(X_j)=\widetilde X_j$, and conjugating $\delta$ with $\tilde\delta$, i.e. such that 
\begin{equation}
\label{conj}
Q\delta=\tilde\delta Q. 
\end{equation}
So, it is in fact more correct to call the symbol of the curve $\Lambda$ at $t$ the equivalence class of $\delta_t$  in the set of degree $-1$ endomorphisms of graded spaces instead of a single degree $-1$ endomorphisms $\delta_t$. Also note that if $\tilde X=X$, then $\delta$ and $\tilde \delta$ are equivalent if and only if they lie in the same orbit under adjoint action of the isomorphisms of $X$ preserving the grading on $\mathfrak {gl}(X)$.

In the case of a curve in Lagrangian Grassmannians $L(V)$ the symplectic form $\sigma$ on $V$ induces the symplectic form $\sigma_t$ on each graded space $\mathrm{gr}V(t)$ as follows: if $\bar x \in V_j(t)$
and $\bar y\in V_{\tilde j}(t)$ with $j+\tilde j=1$, then
$\sigma_t(\bar x,\bar y):=w(x,y)$, where $x$ and $y$ are
representatives of $\bar x$ and $\bar y$ in $\Lambda^{(j)}(t)$ and
$\Lambda^{(\tilde j)}(t)$ respectively; if $j+\tilde j\neq1$, then
$\sigma_t(\bar x,\bar y)=0$. From \eqref{rMC} below it will follow that the symbol $\delta_t$ is not only an endomorphism of $\mathrm gr V(t)$ but also an element of the symplectic algebra $\mathfrak{sp} \bigl(\mathrm{gr}\,V(t)\bigr)$.

We say that a graded space $X=\displaystyle{\bigoplus_{j\in \mathbb Z} X_j}$ with a symplectic form $\sigma$ is a 
\emph{symplectic graded space}, if the flag $\{ X^j\}_{j\in \mathbb Z}$, where $X^j:=\displaystyle{\bigoplus_{i\geq j} X_i}$,  
is a  symplectic flag and after the identification of $X^j/X^{j+1}$ with $X_j$ the symplectic form  induced by $\sigma$ on the graded space $X=\displaystyle{\bigoplus_{j\in \mathbb Z} X^j/X^{j-1}}$ coincides with $\sigma$. Equivalently,  it means that the spaces $X_j\oplus X_{\bar j}$ with $j+\bar j\neq 1$ are isotropic with respect to $\sigma$. To define the equivalence relation on the space of endomorphisms of symplectic graded spaces one has to require that the conjugating isomorphism $Q$, as in \eqref{conj}, preserves the symplectic form. 

Note that the notion of symbol as above can be defined not only for a curve in a Grassmannian (a Lagrangian Grassmannian) via its curve of osculating flags but for any curve of flags (symplectic flags) $\{\Lambda^j(t)\}_{j\in \mathbb Z}$ such that 
\begin{equation}
\label{compat}
\Lambda^j(t)\subseteq\Lambda^{j-1}(t), \quad  (\Lambda^j)^{(-1)}(t)\subseteq \Lambda^{j-1}(t). 
\end{equation}
For this, spaces $\Lambda^{(j)}(t)$ in all previous formulas should be replaced by $\Lambda^j(t)$.   So, the subsequent theory will be developed to the more general case of such curves of flags, which will be called \emph{curves of flags compatible with osculation}.

Fix a flag (a symplectic flag) $\{V^j\}_{j\in Z}$ with $V^{j+1}\subset V^j$ in $V$ and the grading (symplectic grading) of $V$, 
\begin{equation} 
\label{gradV}
V=\displaystyle{\bigoplus_{j\in \mathbb Z}} V_j
\end{equation}
 such that $V^j=V^{j+1}\oplus V_j$. Recall that an endomorphism $A$ of the  graded space $V$ is of degree $k$ if 
$$A(V_j)\subset V_{j+k}.$$
Further, let $G$ denote either $\mathrm{GL}(V)$ or $\mathrm{Sp}(V)$ and by $G^0$ the subgroup of $G$ preserving the chosen flag $\{V^j\}_{j\in Z}$.  The grading on $V$ induces the grading on the Lie algebra $\mathfrak g$: 
\begin{equation}
\label{splitg}
\mathfrak g=\displaystyle{\bigoplus_{k\in\mathbb Z}} \mathfrak g_k, 
\end{equation}
where 
$\mathfrak g_k$ is the space of degree $k$ endomorphism of $V$, belonging to $\mathfrak g$. Given $a\in \mathfrak g$ we will denote by $a_k$ its degree $k$ component, i.e. the $\mathfrak g_k$-component of $A$ with respect to the splitting \eqref{splitg}.

If $\Gamma(t)$ is a smooth lift of  curve of flags $\{\Lambda^j(t)\}_{j\in \mathbb Z}$  compatible with osculation from $G/G^0$ to $G$, then by constructions 
\begin{equation}
\label{rMC}
\delta_t:=\bigl(\Gamma'(t) \Gamma(t)^{-1})_{-1}.
\end{equation}

\begin{definition}
\label{constdef}	 
We say that a curve of flags (symplectic flags)  $\{\Lambda^j(t)\}_{j\in\mathbb Z}$ compatible with osculation is of \emph {constant type}  if for all $t$ symbols $\delta_t$ belong to the same equivalence class. If $\delta$ is a degree $-1$ element of a graded space (a symplectic graded space), then we say that a curve of flags $\{\Lambda^j(t)\}_{j\in{\mathbb Z}}$  compatible with osculation is of \emph {constant type $\delta$}  if for all $t$ the symbol $\delta_t$ is equivalent to $\delta$.  
\end{definition}
\begin{lem}
	\label{constlem}
	The curve of flags $\{\Lambda^j(t)\}_{j\in\mathbb Z}$  is of constant type $\delta$  if and only if for any smooth lift $\Gamma(t)$ of the curve the degree $-1$ component of the structure function $C_{\Gamma} $ of $\Gamma$ lies in the orbit of $\delta$ under the adjoint action of $G^0$ on $\mathfrak g$. 
\end{lem}	
	\begin{proof} Indeed, from  \eqref{rMC} it follows that 
\begin{equation}
\label{MCAs}
\Bigl(C_\Gamma(t)\Bigr)_{-1}=\Bigl(\Omega_{\Gamma(t)} \bigl(\Gamma'(t)\bigr)\Bigr)_{-1}=\Bigl(\Gamma(t)^{-1} \Gamma'(t)\Bigr)_{-1}=\Bigl(\Gamma(t)^{-1} \delta_t \Gamma(t)\Bigr)_{-1}.
\end{equation}
Therefore,  $\delta_t$ is equivalent to $\delta$ if and only if
$\Bigl(C_\Gamma(t)\Bigr)_{-1}$ lies in the orbit of $\delta$ under the adjoint action of $G^0$ on $\mathfrak g$. 
\end{proof}

The set of all equivalence classes of degree $-1$ endomorphisms of a graded space (a symplectic graded space) is finite and all equivalence classes in this case are explicitly described in \cite{flag2}. The finiteness of these equivalence classes implies that a curve of flags (symplectic flags) compatible with osculation   is of constant type in a neighborhood of its generic point. 
Moreover, as it will be shown for completeness later, any equiregular monotonically nondecreasing curve in a Lagrangian Grassmannian is of constant type; the space of equivalence classes of symbols of such curves is in fact in one-to-one correspondence with the tuples $\{\dim\Lambda^{(j)}\} _{j\leq 0}$ or , equivalently, with the set of all Young diagrams (see Proposition \ref{Youngprop} below). Note also that the finiteness of the set of equivalence classes follows in fact from more general result of  E.B. Vinberg \cite{vinb} on finiteness of orbits of degree $-1$ elements of a graded semisimple Lie algebra under the adjoint action of the group of automorphisms of this graded Lie algebra.

\subsection{Flat curves of constant type and their symmetries}
Fix $\delta\in \mathfrak g_{-1}$. 
The ultimate goal of the entire section is to describe  the unified construction of canonical bundle of moving frames for
 all curves of flags (symplectic flags) of constant type $\delta$.

According to the general discussions at the end of the previous section, first we need to distinguish the most symmetric curves within this class. The natural candidate is an orbit under the action of the 
one-parametric group $\mathrm{exp}(t\delta)$  on the corresponding flag variety, for example, the curve of flags $t\rightarrow \{\mathrm{exp}(t\delta) V^j\}_{j\in\mathbb Z}$. Such curve is called the \emph{flat curve} of constant type $\delta$ and will be denoted by $F(\delta)$. As we will see later, this curve is indeed the right candidate for the most symmetric curve in the considered class. The Lie algebra of the symmetry group of $F_\delta$  has an explicit algebraic description.
Since for a parametrized curves the group of non-effective symmetries are important, in this subsection we will focus on the corresponding Lie algebra, referred as the algebra of infinitesimal non-effective symmetries. 

Let
\begin{equation}
\label{0prolong}
\mathfrak u_0(\delta):=\{x\in\mathfrak g_0:[x,\delta]=0\}
\end{equation}
and define recursively
\begin{equation}
\label{kprolong}
\mathfrak u_k(\delta):=\{x\in\mathfrak g_k:[x,\delta]\in \mathfrak u_{k-1}(\delta)\},\quad
k> 0.
\end{equation}

\begin{exer}
Show that 	
\begin{equation}
\label{Tanprol}
\mathfrak u(\delta):=\displaystyle{\bigoplus_{k\geq 0}} \mathfrak u_k(\delta)
\end{equation}
 is a subalgebra of $\mathfrak g$. 
\end{exer}

 The algebra $\mathfrak u(\delta)$ is called the \emph{universal algebraic prolongation of the symbol $\delta$}, and its degree $k$ homogeneous component  $\mathfrak u_k(\delta)$ is called the \emph{$k$th algebraic prolongation of the symbol $\delta$}. 
\begin{rem}
	\label{u0rem}
	Let 
	\begin{eqnarray}
   &~&U_0(\delta)=\{A\in G: \mathrm{Ad}\, A\,\delta=\delta \text{ and }  A(V_j)=V_j \, \forall j\} \label {u_0}\\
	&~&U^0(\delta) =\{A\in G: \bigl(\mathrm{Ad}\, A\,\delta\bigr)_{-1}=\delta \text{ and }  A(V^j)=V^j \,\, \forall j\} \label {u^0}
	\end{eqnarray}
    Then the Lie algebra of $U_{0}(\delta)$ is equal to $\mathfrak u_0(\delta)$ and the Lie algebra of $U^0(\delta)$ is equal to $\mathfrak u_0\oplus\displaystyle{\bigoplus_{k>0} \mathfrak g_k}$.		
	 \end{rem} 
 
\begin{thm}
The algebra of infinitesimal non-effective symmetries $\mathrm{sym}^{\mathrm{ne}}_{F(\delta)}$ of the flat curve $F(\delta)$ of constant type $\delta$ is equal to the algebra $\mathfrak u(\delta)$.
\end{thm}
\begin{proof} By definition,  $y\in \mathrm{sym}^{\mathrm{ne}}_{F(\delta)}$ if and only if $\mathrm {exp}(s y)\in \mathrm{Sym}^{\mathrm{ne}}_{F(\delta)}$ for every $s$ sufficiently close to $0$. The latter means that
\begin{equation}
\label{sym0}
\mathrm {exp}(s y).F_\delta(t)=F_\delta(t)
\end{equation}	
 for every $s$ sufficiently close to $0$ and every $t\in \mathbb R$ or, equivalently
\begin{equation*}
\mathrm{exp}(s y)\circ \mathrm {exp}(t\delta) V^j=	\mathrm {exp}(t\delta) V^j\Leftrightarrow \mathrm {exp}(-t\delta)\circ \mathrm{exp}(sy)\circ \mathrm {exp}(t\delta) V^j=V^j
\end{equation*}
for every $s$ sufficiently close to $0$, every $t\in\mathbb R$ and $j\in \mathbb Z$, which in turn equivalently can be written as 
\begin{equation}
\label{sym1}
\mathrm {exp}(-t\delta)\circ \mathrm{exp}(sy)\circ \mathrm {exp}(t\delta)\in G^0,
\end{equation}
 where, as before, $G^0$ denotes the subgroup of $G$ consisting of all element preserving the flag $\{V^j\}_{j\in\mathbb Z}$. 
 
Differentiating \eqref{sym1} with respect to $s$ at $s=0$, we get
\begin{equation}
\label{sym2}
\mathrm {Ad}\Bigl( \mathrm {exp}(-t\delta)\Bigr) y\in \mathfrak g^0
\end{equation}
where $\mathfrak g^0$ is the Lie algebra of $G^0$. Pass to the Taylor series of the left-hand side of \eqref{sym2}, 
\begin{equation}
\label{Taylor}
\mathrm {Ad}\Bigl( \mathrm {exp}(-t\delta)\Bigr) y=\sum_{k=0}^{\infty}\cfrac{t^k}{k!}\,\,\bigl(\mathrm{ad}(-\delta)\bigr)^k y,
\end{equation}
where the sum is actually finite for every element $y$. Then
we get that \eqref{sym2} is equivalent to 
\begin{equation}
\label{sym3}
\bigl(\mathrm {ad}\,\delta\bigr)^k y\in \mathfrak g^0, \quad k\geq 0.
\end{equation}
Obviously, $\mathfrak g^0=\displaystyle{\bigoplus_{k\geq 0}} \mathfrak g_k$. So, $y$ can be represented  as $y=\displaystyle{\sum_{k\geq 0} y_k}$ with $y_k\in \mathfrak g_k$.
Since $\delta$ has degree $-1$ and all elements of $\mathfrak g^0$ are of nonnegative degree, from \eqref{sym3}
with  $k=1$ it follows that $[\delta,y_0]=0$, i.e. by \eqref{0prolong} we have  $y_0 \in\mathfrak u_0(\delta)$. Further, from \eqref{sym3}
with  $k=2$ it follows that $\mathrm{ad} \delta^2 y_1=0$, which yields that $[\delta, y_1]\in u_0(\delta)$ and so by \eqref{kprolong} we have $y_1 \in\mathfrak u_1(\delta)$. In this way by induction in $k$ one proves that \eqref{sym3} implies that $y_k\in\mathfrak u_k(\delta)$ for all $k\geq 0$, i.e. that $\mathrm{sym}^{\mathrm{ne}}_{F(\delta)}\subset \mathfrak u(\delta)$. 
Finally, the opposite inclusion is valid because \eqref{sym3} implies \eqref{sym0}.
\end{proof}
\begin{rem}
By analogy, one can describe the whole algebra of  infinitesimal symmetries $\mathrm{sym}_{F(\delta)}$ of a curve $F_\delta$ considered as an unparametrized curve (a one-dimensional submanifold of the corresponding flag variety). For this set
\begin{equation*}
\widetilde{\mathfrak u}_{-1}(\delta):=\mathrm{span}\{\delta\}
\end{equation*}
and define recursively
\begin{equation*}
\widetilde {\mathfrak u}_k(\delta):=\{x\in\mathfrak g_k:[x,\delta]\in \widetilde{\mathfrak u}_{k-1}(\delta)\},\quad
k\geq  0.
\end{equation*}
Then  $\mathrm{sym}_{F(\delta)}=\displaystyle{\bigoplus_{k\geq -1}} \widetilde{\mathfrak u}_k(\delta)$. The algebra $\mathrm{sym}_{F(\delta)}$ is in fact the maximal graded subalgebra  among all graded subalgebras of algebras with negative part equal to $\mathrm{span}\{\delta\}$. Moreover the algebra $\mathrm{sym}^{\mathrm{ne}}_{F(\delta)}$ is the largest ideal of $\mathrm{sym}_{F(\delta)}$ concentrated in nonnegative degree and $\mathrm{sym}_{F(\delta)}/\mathrm{sym}^\mathrm{ne}_{F(\delta)}$ is isomorphic to the algebra $\mathrm{sl}_2$.
\end{rem}

\subsection{Construction of canonical frames for curves of constant type $\delta$} Fix $\delta\in \mathfrak g_{-1}$ again. Now we will describe the unified construction of canonical bundle of moving frames for all curves of constant type $\delta$.
Fix a curve $\gamma$. Let $\pi:G\rightarrow G/G^0$ be the canonical projection. By a \emph{moving frame bundle} $B$ over $\gamma$ we mean any subbundle (not necessarily principal) of the $G^0$-bundle $\pi^{-1}(\gamma)\rightarrow \gamma$. 

Let $B(t)=\pi^{-1}\bigl(\gamma(t)\bigr)\cap B$ be the fiber of $B$ over the point $\gamma(t)$. Given any $\Gamma \in B(t)$ consider the tangent space $T_\Gamma B(t)$ to the fiber $B(t)$ at
$\Gamma$. This  space can be identified with the following subspace $W_\Gamma$ of the Lie algebra $\mathfrak g^0$ via the left Maurer-Cartan form $\Omega$: 
\begin{equation}
\label{Lpsi}
W_\Gamma:=\Omega_\Gamma\left(T_\Gamma B(t)\right).
\end{equation}
If $B$ is a principal bundle over our curve, which is a reduction of the bundle $\pi^{-1}(\gamma)\rightarrow \gamma$, then the space $L_\Gamma$ is independent of $\Gamma$ and equal to the Lie algebra of the structure group of the bundle $B$. For our purposes here we need to consider more general class of fiber subbundles of $\pi^{-1}(\gamma)\rightarrow \gamma$. 
To define this class first consider the decreasing filtration  $\{\mathfrak g^k\}_{k\geq 0}$of the graded space $\mathfrak g^0$ where 
$$\mathfrak g^k=\bigoplus_{i\geq k} \mathfrak g_i.$$ 

Given a subspace $U$ of $\mathfrak g^0$ let $U^k=U\cap \mathfrak g^k$. Note that the quotient space $\mathfrak g^k/\mathfrak g^{k+1}$ is naturally identified with $\mathfrak g_k$. Therefore, since the quotient space $U^k/U^{k+1}$ can be considered as a subspace of  $\mathfrak g^k/\mathfrak g^{k+1}$, $U^k/U^{k+1}$ can be naturally identified with a subspace $U_k$ of $\mathfrak g_k$. With this notation, we assign to each subspace $U$ of $\mathfrak g^0$ a graded subspace  $\mathrm {gr}\, U:= \displaystyle{\bigoplus_{k\geq 0} U_k}$ of $\mathfrak g^0$. Note that this space is in general different from the original space $U$. 

The space $\mathrm {gr} W_\Gamma$, where $W_\Gamma$ is as in \eqref{Lpsi}, is called the \emph{symbol} of the bundle $B$ at the point $\Gamma$.

\begin{definition}
\label{symbbunddef}	
 We say that the fiber  subbundle $B$ of  $\pi^{-1}(\gamma)\rightarrow \gamma$ 
 has a \emph{constant symbol $\mathfrak s$} if its symbols at different points
 coincide with $\mathfrak s$. In this case we call $B$ the \emph{quasi-principal subbundle of the bundle  $\pi^{-1}(\gamma)\rightarrow \gamma$ with symbol $\mathfrak s$.} 
\end{definition}


Let $[\delta, \mathfrak g_k]:=\{[\delta, y]:y\in \mathfrak g_k\}$.
\begin{definition}
Let $N=\displaystyle{\bigoplus_{k\geq 0}} N_k$ be a graded subspace of $\mathfrak g^0$, i.e. such that $N_k\subset\mathfrak g_k$. 	We say that $N$ defines a normalization condition if for any $k\geq 0$ the subspace $N_k$ is complementary  to $\mathfrak u_k+[\delta,\mathfrak g_{k+1}]$ in $\mathfrak g_k$.
\begin{equation}
\label{normacond}
\mathfrak g_k= \bigl(\mathfrak u_k(\delta)+[\delta, \mathfrak g_{k+1}]\bigr)\oplus N_k, \,\,k\geq 0.
\end{equation}
In this case we also say that the graded subspace $N_k$ is complementary to $\bigl(\mathfrak u(\delta)+[\delta, \mathfrak g^0]\bigr)\cap\mathfrak g^0$ in $\mathfrak g^0$.
\end{definition}

\begin{thm}
Given a normalization condition $N$,  for any curve $\gamma$ of constant type $\delta$ the set of moving frames $\Gamma(t)$ such that its structure function $C_\Gamma$ satisfies 
\begin{equation}
\label{normframe}
C_ \Gamma(t)-\delta\in N, \quad \forall t,
\end{equation} 
foliates the fiber subbundle of the bundle $\pi^{-1}(\gamma)\rightarrow\gamma$ of constant symbol $\mathfrak u(\delta)$. Moreover, if $N$ is invariant with respect to the adjoint action of the group $\mathrm {Sym}^{\mathrm{ne}}_{F(\delta)}$ of noneffective symmetries of the flat curve $F(\delta)$ of constant type $\delta$, then the resulting bundle is a principal  $\mathrm {Sym}^{\mathrm{ne}}_{F(\delta)}$-subbundle of  $\pi^{-1}(\gamma)\rightarrow\gamma$ and the foliation of moving frames, satisfying \eqref{normframe}, is invariant with respect to the principal  $\mathrm {Sym}^{\mathrm{ne}}_{F(\delta)}$-action.
\end{thm}
\begin{proof}
	We will say that a moving frame $\Gamma(t)$ is \emph{normal up to order $k\geq 0$}, if  its structure function $C_\Gamma$ satisfies
	\begin{equation}
	\label{normframepart-1}
	\bigl(C_ \Gamma(t)\bigr)_{-1}=\delta
	\end{equation}
	for all $t$ and 
	\begin{equation}
	\label{normframepart}
	\bigl(C_ \Gamma(t)\bigr)_i\in N_i
	\end{equation} 
for all $t$ and $0\leq i<k$.

We will construct by induction  the decreasing  sequence of subbundles 
\begin{equation}
\label{fiber_chain}
B^{-1}=\pi^{-1}(\gamma)\supset B^0\subset B^1\supset \ldots 
\end{equation}
such that $B^k$ is the union of all normal up to order $k$ moving frames. Moreover, the bundle $B^k$ has constant symbol $\displaystyle{\bigoplus_{i=0}^k \mathfrak u_i(\delta)\oplus \mathfrak g^{k+1}}$.
	 
Let us describe this inductive procedure.
For $k=0$, the condition \eqref{normframepart} is void and by \eqref{MCAs}  the condition \eqref{normframepart-1} is equivalent to
\begin{equation*}
\Bigl(\Gamma(t)^{-1} \delta_t \Gamma(t)\Bigr)_{-1}= \delta.
\end{equation*}
By Definition \ref{constdef} of curves of constant type $\delta$ such $\Gamma(t)$ exists for any $t$ , i.e. $B^0$ is 
not empty. Moreover, by Remark \ref{u0rem} $B_0$ is the principal reduction of the principal bundle $B_{-1}$ with a structure group $U^0(\delta)$ as in \eqref{u^0} and with the Lie algebra  $\mathfrak u_0\oplus\mathfrak g^1$. In particular, the latter is the symbol of this algebra.

Now assume by induction that the bundle $B_{k-1}$ with the properties above is constructed for some $k\geq 1$ and construct the next bundle $B^k$.  Since, 
by assumptions,  \eqref{normframepart} holds for $i<k-1$, the symbol of $B^{k-1}$ is  $\displaystyle{\bigoplus_{i=0}^{k-1} \mathfrak u_i(\delta)}\oplus \mathfrak g^{k}$, and   the spaces $N_i$ and $\mathfrak u_i(\delta)$ intersect trivially by \eqref{normacond}, we have that if two  normal up to order $k-1$ moving frames $\Gamma$ and $\widetilde\Gamma$  pass through the same point at $t=t_0$, i.e. $\Gamma(t_0)=\widetilde\Gamma(t_0)$, then 
\begin{eqnarray}
&~&	\Bigl(C_\Gamma(t_0)\Bigr)_i=\Bigl(C_{\widetilde\Gamma}(t_0)
\Bigr)_i, \quad 0\leq i\leq k-2;\label{uptok-2}\\
&~&	\Bigl(C_\Gamma(t_0)\Bigr)_{k-1}=\Bigl(C_{\widetilde\Gamma}(t_0)\Bigr)_{k-1} \quad \mod \mathfrak u_{k-1}(\delta). \label{uptok-1}
\end{eqnarray}
In other words, for $0\leq i\leq k-2$  the degree $i$ component of the structure function of a normal up to order $k-1$ frame, passing through a point $b\in B^{k-1}$, depends not on the frame but on $b$ only, while the degree $k-1$ component depends on $b$ modulo $\mathfrak u_{k-1}(\delta)$. For the latter case, let us denote by $\xi_{k-1}(b)$ the corresponding element of $\mathfrak g_{k-1} \mathrm {mod}\, \mathfrak u_{k-1}$ or, equivalently, of $\mathfrak g_{k-1}/\mathfrak u_{k-1}(\delta)$.

Now let us explore how the function $R_{k-1}$ changes along the fiber of $B^{k-1}$.
Denote by $R_a$ the right translation by $a$ in the group $G$.
\begin{lem}
The following identity holds
\begin{equation}
\label{structtrans}
\xi_{k-1}\bigl(R_{ \mathrm{exp} \,x}b\bigr)=\xi_{k-1}(b)+[\delta, x]_{k-1} \quad \mathrm {mod}\,\, \mathfrak u_{k-1}(\delta), \quad \forall  x\in g^k,
\end{equation} 
or, equivalently,
\begin{equation}
\label{structtranscase}
\xi_{k-1}\bigl(R_{ \mathrm{exp}\, x}b\bigr)=\begin{cases}\xi_{k-1}(b)+[\delta, x]\quad \mathrm {mod}\,\, \mathfrak u_{k-1}(\delta), &  x\in g_k;\\
\xi_{k-1}(b) \quad \mathrm {mod}\,\, \mathfrak u_{k-1}(\delta), &   x\in g^{k+1}.\\
\end{cases}
\end{equation} 
\end{lem}
\begin{proof}
 We use the following equivariance property of the left Maurer -Cartan form:

\begin{equation}
\label{MCequiv}
R_a^*\Omega=\Bigl(\mathrm{Ad}\, a^{-1}\Bigr)\Omega, \quad \forall a\in G.
\end{equation}
Let us prove \eqref{MCequiv} for completeness. Indeed,
using the definition of the left Maurer-Cartan form given by \eqref{MC} and the fact that any left translation commutes with the right translation, we have 
$$R_a^*\Omega_b (y)=\Omega _{ba} \bigl((R_a)_*y\bigr)=
\Bigl((L_{(ba)^{-1}})_*\circ  (R_a)_*\Bigr)(y)=\Bigr((L_{a^{-1}})_*\circ(R_a)_*\circ (L_{b^{-1}}\Bigr)_*(y)=\Bigl(\mathrm{Ad}\, a^{-1}\Bigr)\Omega_b(y),$$
for any $y\in \mathfrak g$, which proves \eqref{MCequiv}. 

Take a moving frame $\Gamma$ over $\gamma$ and  consider the moving frame $R_{ \mathrm{exp} \,x}(\Gamma)$.  Using formula \eqref{MCequiv}  and the Taylor expansion as in \eqref{Taylor} with $\delta=x$, one can  relate the structure functions of the frames  $\Gamma$ and $R_{ \mathrm{exp} \,x}(\Gamma)$ in the following way:

\begin{equation}
\label{expan}
C_ {R_{ \mathrm{exp} \,x}(\Gamma)}(t)= C_\Gamma(t)+
\sum_{k=1}^{\infty}\cfrac{1}{k!}\,\,\bigl(\mathrm{ad}(-x)\bigr)^k C_{\Gamma}(t),
\end{equation}
where the sum is actually finite. Comparing the homogeneous components of degree less than $k-1$  in both sides of \eqref{expan}, we get that if $x\in \mathfrak  g^k$ and the moving frame $\Gamma$ is normal up to the order $k-1$,  then the moving frame $R_{\mathrm{exp} \,x(\Gamma)}$ is normal up to the order $k-1$. Therefore, comparing the degree $k-1$ components in both sides of \eqref{expan} and using \eqref{normframepart-1}, we can replace the structure function by the function $\xi_{k-1}$ evaluated at the appropriate points to obtain \eqref{structtrans}.
\end{proof}

Now, from \eqref{structtranscase} and the fact that $N_{k-1}$ is complementary to $[x, \mathfrak g_k]$ mod $\mathfrak u_{k-1}(\delta)$ it follows that for $b\in B^{k-1}$,  one can find  $x\in \mathfrak g_k$ such that

\begin{equation}
\label{almostnorm}
\xi_{k-1}(R_{\mathrm{exp} \,x} b)\in N_{k-1} \,\, \mathrm{mod}\,\, \mathfrak u_{k-1}(\delta). 
\end{equation}

Moreover, since the degree $k-1$ component of the symbol of the bundle $B^{k-1}$ is equal to $\mathfrak u_{k-1}(\delta)$ one can find a normal up to order $k$ moving frame $\Gamma$ passing through $R_{\mathrm{exp} x} b$ which implies that $R_{\mathrm{exp} x} b\in B^k$, i.e $B^k(t_0)$ is not empty, where $t_0$ is such that $\pi(b)=\gamma(t_0)$.

Now, if $b\in B^k$ , then \eqref{structtranscase} implies that $R_{\mathrm{exp}\,x} b\in B^k$ for $x\in \mathfrak g_k$ if and only if 
\begin{equation}
\label{freedom}
[x, \delta]\in N_{k-1}\oplus\mathfrak u_{k-1}(\delta).
\end{equation}
Since by \eqref{normacond} $N_{k-1}$ is transversal to $\mathfrak u_{k-1}(\delta)+[\delta, \mathfrak g_k]$ in $\mathfrak g_{k-1}$, the relation \eqref{freedom} implies that $[\delta, x]\in \mathfrak u_{k-1}(\delta)$, hence $x\in \mathfrak u_k(\delta)$. This implies that $B_k$ is the fiber subbundle of $B^{-1}$ with constant symbol  $\displaystyle{\bigoplus_{i=0}^k \mathfrak u_i(\delta)\oplus \mathfrak g^{k+1}}$, which concludes the proof of the induction step.

Since there exists an integer $m$ such that $\mathfrak g_i=0$ for all $i\geq m$, the sequence of bundles \eqref{fiber_chain} will be stabilized, i.e $B^i=B^m$ for all $i\geq m$. Moreover, the normal up to order $m$ moving frames  will foliate $B^m$, because for any point $b\in B^m$ the structure function of any normal up to order $m$ moving frame and therefore the tangent line to such a moving frame will depend  on the point $b$ only. So, there is a unique normal up to order $m$ moving frame which passes through $b$. So, $B^m$ is the desired bundle of moving frames, which  completes the proof of the first part of the theorem.

The moving frames, which are normal up to order $m$, will be called simply \emph{normal}.
If $\mathcal N$ is invariant with respect to the adjoint action of $\mathrm{Sym}^{\mathrm{ne}}_{F(\delta)}$, then by \eqref{MCequiv}, if the moving frame $\Gamma$ is normal, then for any $u\in \mathrm{sym}^{\mathrm{ne}}_{F(\delta)}$ the moving frame $R_{exp\, u} (\Gamma)$ is normal as well.
Hence the bundle  $B^m$ is a principal $U^0( \delta)$-bundle and the foliation of normal moving frames is invariant with respect to $R_{exp\, u}$, which  completes the proof of the second part of our theorem.
\end{proof}

\section{Application to differential geometry of monotonic parametrized curves in Lagrangian Grassmannians}

Now we apply the general algebraic theory, developed in the previous section, to construct the canonical bundle of moving frames for monotonic parametrized curves in Lagrangian Grassmannians. For this we will  
first classify all possible symbols of their osculating flags, compute their algebraic prolongation, and find the natural invariant normalization conditions. 

\subsection{Classification of symbols of monotonic curves in Lagrangian Grassmannians}
Let  $\Lambda(t)$ be a parametrized equiregular curve in Lagrangian Grassmannians $L(V)$. As in formula  \eqref{filt} of subsection \ref{oscflagsec}, let $\{\Lambda^{j}\}_{j\in Z}$ be the osculating flag. 
We do not lose much by assuming that there exists a negative   integer $p$
such that 
\begin{equation}
\label{amplep}
\Lambda^{(p)}(t)=V
\end{equation}
Otherwise, if
$\Lambda^{(p-1)}(t)=\Lambda^{(p)}(t)\subset V$, then the subspace
$\tilde V=\Lambda^{(p)}(t)$ does not depend on $t$ and one can work with
the curve $\Lambda(t)/\widetilde V^\angle$ in the symplectic space $\widetilde V/\widetilde V^\angle$
instead of $\Lambda(t)$. 

\begin{definition}
	\label{ampledef}
	The curve $\Lambda$ is called ample, if for any $t$ there exists $p$ for which \eqref{amplep} holds.
\end{definition}

Let $D$ be the Young diagram of the curve $\Lambda$ (see Definition \ref{Youngdef}).
Let the length of the rows of $D$ be $p_1$ repeated $r_1$ times,
$p_2$ repeated $r_2$ times, $\ldots$, $p_s$ repeated $r_s$ times
with $p_1>p_2>\ldots>p_s$. \emph {The reduction or the reduced Young diagram  of the Young diagram $D$} is the Young diagram $\Delta$, consisting of $k$ rows such that
the $i$th row has $p_i$ boxes.

Make the mirror reflection  of the Young diagram $\Delta$
with respect to its left vertical edge . Denote the skew-diagram obtained by union of this mirror reflection and $\Delta$ by $\widetilde\Delta$. Denote by $r$ and $l$ the right and left shifts on $\widetilde \Delta$, respectively. In other words given a box $a$ of $\widetilde\Delta$ denote by $r(a)$ and $l(a)$  the boxes next to $a$ to the right and to the left, respectively, in the same row of $\Delta$.   Also let $m:\widetilde \Delta\rightarrow \widetilde\Delta$ be the mirror reflection with respect to the left vertical edge of the diagram $\Delta$, i.e. the map sending a box $a$ of $\widetilde \Delta$ to  the box which is mirror-symmetric to $a$ with respect to this left edge.
 
We say that the basis $\{E_a\}_{a\in \widetilde\Delta}$ of $V$, where for a box $a$ from the $i$th row of $\widetilde \Delta$ $E_a$ is the tuple of $r_i$ vectors in $V$, $E_a=(e_a^1,\ldots e_a^{r_i})$, forms a \emph{Darboux basis indexed by the diagram $\widetilde D$}, if  any vector from the tuple $E_a$ is skew-orthogonal to any vector from the tuple  $E_b$ for $b\neq m(a)$ and

\begin{equation}
\label{Darboux}
\sigma(e_{m(a)}^i, e_a^j)=\delta_{ij} ,\quad  a \in \Delta.
\end{equation}

Given a tuple of vectors $E$ in $V$ and an endomorphism $X$ of $V$, by $XE$ we denote the tuple of vectors obtained by applying $X$ to vectors of $E$. If $Y$ is a matrix  with the same number of rows as the number of vectors in E, then  by $EY$ we mean the new tuple with the $j$th vector equal to the linear combination of vectors $E$ with coefficients in $j$th column of $Y$.
 
\begin{exer} 
	Show that  the map $X\in\mathfrak{sp}(X)$ if and only if  it has the representation in a Darboux basis   $\{E_a\}_{a\in \widetilde\Delta}$
\begin{equation}
\label{matrixrep}
X(E_a)=\sum_{b\in \widetilde \Delta} E_b X_{ba}
\end{equation}
such that 
\begin{eqnarray}
&~& X_{ab}=X^T_{ m(b) m(a)}, \quad a\in\Delta,b \in m(\Delta)\label{ab1}\\
&~&  X_{ab}=-X^T_{ m(b) m(a)}\ a,b\in\Delta\label{ab2}
\end{eqnarray}
\end{exer}

If we denote
\begin{equation}
\label{epsilon}
\varepsilon(a)=\begin{cases}
-1, & a\in\Delta;\\
1,& a \in m(\Delta)
\end{cases}
\end{equation}
then \eqref{ab1} and \eqref{ab2} can be written as 
\begin{equation}
\label{abgen}
 X_{ab}=-\varepsilon(a)\varepsilon(b)X^T_{ m(b) m(a)}
\end{equation}

\begin{prop}
	\label{Youngprop}
	Any monotonic equiregular ample curve with Young Diagram  $D$ in a Lagrangian Grassmannian has the unique symbol represented by the endomorphism $\delta$ acting on a  Darboux basis $\{E_a\}_{a\in \widetilde\Delta}$ as follows:
	\begin{equation}
	\label{normformdelta}
	\delta(E_a)=\varepsilon(a)E_{r(a)},
\end{equation}
where $\varepsilon (a)$ is defined in \eqref{epsilon}.
In particular, there is one-to-one correspondence between 
the set of Young diagrams and the set of symbols of monotonic curve in Lagrangian Grassmannians. 
\end{prop}

\begin{proof}
Let $V_j$ and $\mathrm{gr} V(t)$ be as \eqref{gradj} and \eqref{grad}. 
Let $\delta_t$ be the symbol of the curve of osculating  flags
$\{\Lambda^{j}\}_{j\in Z}$ at $t$. 

\begin{lem}
	\label{classlem}
	If $\Lambda$ satisfies the assumption of Proposition \ref{Youngprop}, then  
	for any
	$j\geq 0$ the map $\delta_t^{2j+1}:V_{j}(t)\rightarrow V_{-j-1}(t)$ is an isomorphism.  
\end{lem}

\begin{proof}
	Let $\sigma_t$ be the natural symplectic form on $\mathrm{gr} V_t$ induced by the symplectic form on $V$ as described in subsection \ref{oscflagsec}. Note that the bilinear form $(x,y)\mapsto \sigma_t(\delta_t x, y)$ on $V_0(t)$ is symmetric and nondegenerate, as follows from \eqref{sympalg} and the construction of the spaces $\Lambda^{(1)}(t)$.
	From the fact that the osculating flags are symplectic it follows that  
	\begin{equation}
	\label{oms}
	\sigma_t(\delta_t^{2j+1} x,  y)=(-1)^s\sigma_t(\delta_t^{2j+1-s} x, \delta_t^s
	y)
	\end{equation}
	for all $ x, y\in V_{j}(t)$. In particular, for $s=j$ 
	$$\sigma_t(\delta_t^{2j+1} x,  y)=(-1)^j\sigma_t(\delta_t^{j+1} x, \delta_t^j y)$$
	for all $ x, y\in V_{j}(t)$. This implies that the bilinear form  $\sigma_t(\delta_t^{2j+1} x,  y)$ on $V_j(t)$ is
	symmetric and the  desired condition that  $\delta_t^{2j+1}:V_{j}(t)\rightarrow V_{-j-1}(t)$ is an isomorphism is equivalent to the fact that this form is nondegenerate. Besides, since $\delta_t^j:V_j(t)\mapsto V_0(t)$ is injective (see Remark \ref{surinrem}) the latter statement is equivalent to the fact that the symmetric bilinear form $\sigma_t(\delta_t x, y)$ (and hence the corresponding quadratic form) is nondegenerate on the subspace $\delta_T^j(V_j)$ of $V_0(t)$. Finally, since by the assumption the curve $\Lambda$ is monotonic, the quadratic form $\sigma_t(\delta_t x, y)$ is positive definite on $V_0(t)$, and hence its restriction to any subspace of $V_0(t)$ is positive definite and hence  nondegenerate. This completes the proof of the lemma.
\end{proof}

Let $\rho_i$ be the last (i.e. the most right) box of the $i$th row of the reduced Young diagram  $\Delta$ of $D$ and let again $p_i$ is the number of boxes in the $i$th row of $\Delta$. Let $E_{m(\rho_1)}(t)$ be a basis of $V_{p_1}(t)$ orthonormal with 
respect to the inner product  $(-1)^{p_1}\sigma_t(\delta_t^{2p_1+1}x,y)$. Then for $0\leq s< p_1-p_2-1$ set $$E_{r^s\bigl(m(\rho_1)(t)\bigr)}:=\delta_t^s E_{m(\rho_1)}(t).$$
Note that by \eqref{oms} the tuples $\delta_t^s E_{m(\rho_1)}(t)$ 
are orthonormal with respect to the inner product $$(-1)^{p_1-s}\sigma_t(\delta_t^{2(p_1-s)+1} x, x).$$ Further,  let $E_{m(\rho_2)}(t)$ be a completion of the tuple 
$\delta_t^{p_2-p_1-1} E_{m(\rho_1)}(t)$ to an orthonormal 
basis of $V_{p_2}(t)$ with respect to the inner product 
$(-1)^{p_2}\sigma_t(\delta^{2p_2+1}x,x)$. So, we defined $E_a(t)$
for all $a$ not located to the right of $\rho_2$ in $m(\Delta)$. 

In the same way,  by applying $\delta$ and completing the constructed tuples to orthonormal bases  of the corresponding $V_j(t)$ with respect to the corresponding inner product, we can define $E_a$ for all $a\in m(\Delta)$ such  that $\delta_t E_a(t)$ 
is either equal to $E_{r(a)}(t)$ or is subtuple of $E_{r(a)}(t)$. Further, for the box $a$ in the $j$th column of the diagram 
$\Delta$ set $E_a(t):=(-1)^{j-1} \delta_t^{2j-1}E_{m(a)}(t)$. 
Then again from \eqref{oms} it  follows that $\{E_a(t)\}_{a\in \widetilde \Delta}$ is a Darboux frame of $\mathrm {gr} V(t)$. Moreover, by constructions $\delta_t$ acts on the basis 
$\{E_a(t)\}_{\widetilde D}$ as in \eqref{normformdelta}, which proves the statement.
\end{proof}

\subsection{Calculation of the universal algebraic prolongation of $\delta$}
We assume that $\delta$ has the form \eqref{normformdelta} in a Darboux basis $\{E_a\}_{a\in \widetilde\Delta}$.
First, let us calculate the commutator of $\delta$ with $X\in \mathfrak {sp} (V)$ in the Darboux basis $\{E_a\}_{a\in \widetilde\Delta}$. We need it both for calculation of the algebraic prolongation and for the choice of the normalization conditions. 
\begin{lem}
	\label{comlem}
	If 
	\begin{equation}
	\label{delX}
	[\delta, X](E_a)=\sum_{b\in\widetilde\Delta} E_b Y_{ba},
	\end{equation}
	 then
	\begin{equation}
	\label{commeq}
	Y_{ba}=\varepsilon \bigl(l(b)\bigr) X_{l(b)a}- \varepsilon (a)X_{b\, r(a)},
	\end{equation}
	where the terms involving non-existing boxes of the diagram $\widetilde \Delta$ are considered to be equal to zero. 
\end{lem}
\begin{proof}
	Using \eqref{normformdelta} and \eqref{matrixrep}, we have
	
	$$[\delta, X](E_a)=\delta\circ X(E_a)-X\delta(A)=
	\delta\Bigl(\sum_{b\in\widetilde\Delta} E_b X_{ba}\Bigr)-\varepsilon(a) X(E_{r(a)})=$$
	$$\sum_{b\in\widetilde\Delta} \varepsilon (b)E_{r(b)} X_{ba}-\varepsilon (a)\sum_{b\in  \widetilde \Delta} E_b X_{b\, r(a)}=\sum_{b\in  \widetilde \Delta}E_b Y_{ba}$$
		with $Y_{ba}$ as in \eqref{commeq}, which completes the proof of the lemma.
\end{proof}
\begin{prop}
	\label{prolongprop}
	The following holds:
	\begin{equation}
	\label{prolongeq}
	\mathfrak u(\delta)=\mathfrak u_0(\delta)\cong\bigoplus_{i=1}^s \mathfrak {so}_{r_i} 
	\end{equation}
	In more detail, $\mathfrak u(\delta)$ consists of all $X\in \mathfrak g^0$ such that if $X$ is represented in the Darboux frame  $\{E_a\}_{a\in \widetilde\Delta}$ by \eqref{matrixrep}, then the only possibly nonzero matrices $X_{ba}$ are when $a=b$ and in this case each $X_{aa}$ is skew-symmetric and $X_{l(a) l(a)}=X_{(a,a)}$ for any box $a \in \widetilde \Delta$, which is not the first box of a row.
\end{prop}
\begin{proof}
First let us  describe all $X$ in $\mathfrak g^0$ (and not necessarily in $\mathfrak g_0$) that commute with $\delta$. This will allow us to calculate $\mathfrak u_0$ and will be used to prove that $\mathfrak u_i=0$ for $i>0$.

\begin{lem} 
\label{commlem}	
If $X\in \mathfrak g^0$ and $[\delta, X]=0$, then 
the only possibly nonzero matrices $X_{ba}$ in the representation \eqref{matrixrep} are when $a=b$ and in this case each $X_{aa}$ is skew-symmetric and $X_{l(a) l(a)}=X_{(a,a)}$ for any box $a \in \widetilde \Delta$, which is not the first box of a row.
\end{lem}	
\begin{proof}	
Let, as before, $\rho_i$ be the last (i.e. the most right) box in the $i$th row of $\widetilde\Delta$. Since $r(\sigma)$ does not exist, by \eqref{commeq} we get that  the condition $Y_{b\rho_i}=0$ implies that 
\begin{equation}
\label{lbsigma}
X_{l(b)\rho_i}=0.
\end{equation}
This means that $X_{b\rho_i} =0$ for all $b\neq \rho_j$. Since $X\in\mathfrak g^0$, $X_{\rho_j \rho_i}=0$ if $i<j$.  Also, if $j>i$, then  \eqref{ab2} implies that $X_{\rho_j \rho_i}=-X_{m(\rho_i) m(\rho_j)}^T$ and the latter is $0$ from the condition $X\in\mathfrak g^0$ again. 
So, we got that
\begin{equation}
\label{bsigma}
X_{b\rho_i} =0, \quad b\neq \rho_i.
\end{equation}
	
Using relations \eqref{lbsigma}  and  \eqref{commeq}  we get  that the condition $Y_{l(b) l(\rho_i)}
=0$ implies that
$$X_{l^2(b)l(\rho_i)}=0.$$
In the same way, by induction we will get that if $[\delta, X]=0$, then 
$$X_{l^{i+1}(b)l^i(\rho_i)}=0,$$
which together with \eqref{bsigma} yields
$$X_{ba}=0, \quad b\neq a.$$

Now  treat the case $b=a$. From  \eqref{commeq} the condition $Y_{r(a)a}=0$, where $a$ is not the last box of a row $\widetilde \Delta$, is equivalent to 
$X_{aa}=X_{r(a) r(a)}, $ which implies that $X_{aa}=X_{bb}$ if boxes $a$ and $b$ lie in the same row of $\widetilde \Delta$.
Since boxes $a$ and $m(a)$ lie in the same row, we get from this and \eqref{ab2} that 
$$X_{aa}=X_{m(a)m(a)}=-X_{aa}^T,$$
i.e. $X_{aa}$ is skew-symmetric, which completes the proof of the lemma.
\end{proof}

From the previous lemma we get that $\mathfrak u_0$ consists of all $X\in \mathfrak g_0$ such that  the only possibly nonzero matrices $X_{ba}$ in the representation \eqref{matrixrep} are when $a=b$ and in this case each $X_{aa}$ is skew-symmetric and $X_{l(a) l(a)}=X_{(a,a)}$ for any box $a \in \widetilde \Delta$, which is not the first box of a row.

Also, from the previous lemma it follows that  in order to complete the proof of the proposition it is enough to prove that $\mathfrak u_1=0$, because there is no nontrivial elements of degree $\geq 2$ in $\mathfrak g$ , which commute with $\delta$.  

To calculate $\mathfrak u_1$,  let $Y_{ba}$ be as in \eqref{delX}. If  $\rho_i$ is  the last box of the $i$ row  of $\widetilde \Delta$  and $[\delta,X]\in\mathfrak u_0$, then  $Y_{b\rho_i}=0$ form $b\neq \rho_i$. Hence by \eqref{commeq} we have \eqref{lbsigma} for $b\neq \rho_i$ and by exactly the same arguments as in Lemma \ref{comlem} and the assumption that $X\in \mathfrak g_1$ one gets that 
\begin{equation}
\label{bnotla}
X_{ba}=0, \quad b\neq l(a).
\end{equation}

Now, applying \eqref{commeq} for $b=a$, we get 
\begin{equation}
\label{Caa}
Y_{aa}=\varepsilon \bigl(l(a)\bigr) X_{l(a)a}- \varepsilon (a)X_{a\, r(a)},
\end{equation}
where $Y_{aa}$ is the same skew-symmetric matrix for all $a$ on the same row of $\widetilde \Delta$. For $a=\rho_i$ formula \eqref{Caa} implies 
\begin{equation}
\label{Caa1}
X_{l(\rho_i)\rho_i}=\varepsilon \bigl(l(\rho_i)\bigr)  Y_{\rho_i\rho_i}
\end{equation}

Now use \eqref{Caa} for  $a=l(\rho_i)$
\begin{equation}
\label{Caa2}
Y_{l(\rho_i)l(\rho_i) }=\varepsilon \bigl(l^2(\rho_i)\bigr) X_{l^2(\rho_i) l(\rho_i)}- \varepsilon\bigl(l(\rho_i)\bigr) X_{l(\rho_i) \rho_i}
\end{equation}
Substituting \eqref{Caa1} into \eqref{Caa2} and using that $Y_{l(\rho_i)l(\rho_i) }=Y_{\rho_i\rho_i}$, we get 
\begin{equation*}
X_{l^2(\rho_i)l(\rho_i)}=2\varepsilon \bigl(l^2(\rho_i)\bigr)  Y_{\rho_i\rho_i}.
\end{equation*}
Continuing by induction we get 
\begin{equation}
\label{Caai}
X_{l^j(\rho_i)l^{j-1}](\rho_i)}=j\varepsilon \bigl(l^j(\rho_i)\bigr)  Y_{\rho_i\rho_i},
\end{equation}
which implies that for every  $a$ in the $i$th row  of $\widetilde\Delta$ the  matrix $X_{l(a) a}$ is a nonzero multiple of the same skew-symmetric matrix $Y_{\rho_i\rho_i}$.  

Now  assume that $a_i$ is the first box in the $i$th row of $\Delta$. Then $l(a_i)=m(a_i)$ hence by \eqref{ab1} we have 
$$X_{l(a_i) a_i}=X_{m(a_i) a_i}=X_{m(a_i) a_i}^T=X_{l(a_i) a_i}^T,$$
i.e. $X_{l(a_i) a_i}$ is simultaneously symmetric and skew-symmetric and hence it is equal to zero. This implies that $X_{l(a) a}=0$, which together with \eqref{bnotla} yields that $X=0$. So, we proved that $\mathfrak u_1=0$ and hence $u_i=0$ for all $i\geq 2$. The proof of the proposition is completed.
\end{proof}
 
 \begin{rem}
 	\label{Orem}
 As a matter of fact, $\mathfrak u_0(\delta)$ can be found without calculations from the fact that each space $V_j(t)$, $j\geq 0$  
 is endowed with the Euclidean structure given by the quadratic form  $\sigma_t(\delta_t^{2j+1} x,  x)$. Also, from this and the fact that $\mathfrak u(\delta)=\mathfrak u_0(\delta)$ it is obvious
 that $\mathrm{Sym}^{\mathrm{ne}}_{F(\delta)}\cong O_{r_1}\times \ldots\times O_{r_s}$.  The adjoint action of this group on $\mathfrak g$ can be described as follows: If $U=(U_1, \ldots, U_s)$, where $U_i\in O_{r_i}$ and $X\in\mathfrak g$, then 
 \begin{equation}
 \label{ADact}
 (\mathrm{Ad}\, U X)_{ba}=U_j X_{ba}  U_i^{-1},
 \end{equation}
 where $a$ and $b$ are in the $i$th and $j$th rows of $\widetilde \Delta$, respectively.
 \end{rem}

\subsection {Calculation of $[\delta, \mathfrak g^0]$}
Before choosing the normalization condition, i.e. a graded subspace complementary to  $\mathfrak u(\delta)+[\delta, \mathfrak g^0]\cap\mathfrak g^0$  in $\mathfrak g^0$, we have to describe the space  $[\delta, \mathfrak g^0]\cap \mathfrak g^0$. The following notation will be useful for this purpose. Given $Y\in \mathfrak g^0$, which has the form 
$Y=\displaystyle{\sum_{b\in\widetilde\Delta}} E_b Y_{ba}$, let 

\begin{equation}
\label{ddef}
D(Y)_{ba}:=Y_{ba}+\cfrac{\varepsilon\bigl(l(b)\bigr)}{\varepsilon\bigl(l(a)\bigr)} Y_{l(b) l(a)}+\cfrac{\varepsilon\bigl(l(b)\bigr)}{\varepsilon\bigl(l(a)\bigr)}\cfrac{\varepsilon\bigl(l^2(b)\bigr)}
{\varepsilon\bigl(l^2(a)\bigr)}Y_{l^2(b) l^2(a)}+\ldots=
	\sum_{j\geq 0}\left(\prod_{s=1}^{j}\cfrac{\varepsilon\bigl(l^s(b)\bigr)}{\varepsilon\bigl(l^s(a)\bigr)} \right)Y_{l^j(b) l^j(a),}
\end{equation}
where the sum is finite as we reach the first box of a row after finite number of applications of $l$.

\begin{prop}
	\label{coboundimprop}
$Y\in [\delta, \mathfrak g^0]\cap \mathfrak g^0$ if and only if 
for every last box $\rho$ of the diagram $\widetilde \Delta$ and every box $b\in \widetilde \Delta$ that is not higher than  $\rho$ in $\widetilde \Delta$ the following identity holds

\begin{equation}
\label{coboundeq}
D(Y)_{b\rho}=0.
\end{equation}

\end{prop}

\begin{proof}
Let $(a,b)$ be a pair of boxes of $\widetilde\Delta$ such that $b$  is not to the right and not higher than  $a$. Note that from \eqref{abgen} an element $X\in\mathfrak g^0$ is determined uniquely from the knowledge  of  $X_{ba}$ for all such pairs.
	
If $Y=[\delta, X]$ for some $X\in \mathfrak g^0$, then 	
applying  \eqref{commeq} to pairs of boxes $(b,a)$, $\bigr(l(a), l(b)\bigl)$ ,$\bigr(l^2(a), l^2(b)\bigl), \ldots$   
we get the following  chain of identities:

\begin{equation}
\label{chainY}
\begin{split}
~& 	Y_{ba}=\varepsilon \bigl(l(b)\bigr) X_{l(b)a}- \varepsilon (a)X_{b\, r(a)},\\
~&      Y_{l^2(b)l^2(a)}=\varepsilon \bigl(l^2(b)\bigr) X_{l^2(b)l(a)}- \varepsilon \bigl(l(a)\bigr)X_{l(b)\, a)},\\
~&      Y_{l^3(b)l^3(a)}=\varepsilon \bigl(l^3(b)\bigr) X_{l^2(b)l(a)}- \varepsilon \bigl(l^2(a)\bigr)X_{l^2(b)\, l(a)},\\
~&\vdots\\
 ~&Y_{l^{j-1}(b)l^{j-1}(a)}=\varepsilon \bigl(l^j(b)\bigr) X_{l^{j}(b)l^{j-1}(a)}- \varepsilon \bigl(l^{j-1}(a)\bigr)X_{l^{j-1}(b)\, l^{j-2}(a)},\\
  ~&Y_{l^j(b)l^j(a)}=- \varepsilon \bigl(l^{j}(a)\bigr)X_{l^{j}(b)\, l^{j-1}(a)},
 \end{split} 
  \end{equation}
where $j$ is such that $l^j(b)$ is the first box in the corresponding row of $\widetilde \Delta$. Note that by the assumption on the location of $b$ with respect to $a$ all indices appearing in \eqref{chainY}, except maybe $r(a)$, are well defined. Eliminating  $X_{l(b)a}$ by taking an appropriate linear combination of the first identities in \eqref{chainY}, then eliminating $X_{X_{l^2(b)l(a)}}$ from the resulting combination by adding the third identity of \eqref{chainY}, and continuing this successive eliminating procedure 
we get that 
\begin{equation}
\label{DXY}
\varepsilon(a) X_{b r(a)}=-D(Y)_{ba}
\end{equation}
This implies \eqref{coboundeq} in the case when $a=\rho$, i.e.  the last box of a row in $\widetilde \Delta$, which proves necessity of \eqref{coboundeq}.

To prove  sufficiency, given $Y\in \mathfrak g^0$, satisfying \eqref{coboundeq}, define $X$ such that it satisfies \eqref{DXY} for all pair $(a,b)$ , where $a$ is not the last box in a row and  $b$  is not to the right and not higher than  $r(a)$, and also such that \eqref{abgen} holds. It can be shown  that conditions  \eqref{DXY} and \eqref{abgen} are consistent in the case when $a$ and $b$ lie in the same row, so such $X$ indeed can be constructed and $X\in \mathfrak g^0$. Moreover, by reversion of the procedure of going from \eqref{chainY} to \eqref{DXY} we can show that $[\delta, X]=Y$, which completes the proof of sufficiency.
\end{proof}

Now in order to choose a normalization condition the following lemma is useful:

\begin{lem}
	\label{complementlem}
	Assume that $\rho$ is the last box of a row of $\widetilde \Delta$ and $b$ be the $k$th box in the same row (from the left). Then if $Y\in \mathfrak g^0$ the matrix $D(Y)_{b\rho}$ is symmetric if $k$ is odd and skew-symmetric if $k$ is even. 
	\end{lem}
\begin{proof}
	In the considered case in the  sum \eqref{ddef} defining $D(Y)_{b\rho}$ the terms are subdivided into pairs satisfying relation \eqref{abgen}. Indeed, for any $j$, $0\leq j\leq k$ it is easy to show that 
	
	$$m\bigl(l^j(b)\bigr)=l^{k-1-j}(\rho), \quad m\bigl(l^j(\rho)\bigr)=l^{k-1-j}(b).$$
	Hence, by \eqref{abgen},
	
	\begin{equation}
	\label{rowinv}
	Y_{l^j(b) l^j(\rho)}=-\varepsilon\bigl(l^j(b)\bigr)\varepsilon \bigl(l^j(\rho)\bigr)\bigl(Y_{l^{k-1-j}(b) l^{k-1-j}(\rho)}\bigr)^T
	\end{equation}
	
Assume that the number of boxes in the considered row of $\widetilde D$ is equal to $2p$ Consider the following two cases separately

{\it  Case 1.} Assume that $k\leq p$. Then for all $j$, $0\leq j\leq k-1$, 
\begin{equation}
\label{case11}
l^j(b)\in m(\Delta), \quad l^j(\rho)\in\Delta
\end{equation}

Hence, $\varepsilon\bigl(l^j(b)\bigr)=-\varepsilon\bigl(l^j(\rho)\bigr)$. Therefore,  
by \eqref{ddef}
\begin{equation}
\label{case12}
D(Y)_{b\rho}=\sum_{j=0}^{k-1}(-1)^j Y_{l^j(b)l^j(\rho)}
\end{equation}
 and by \eqref{rowinv} 
 \begin{equation}
 \label{rowinv1}
 Y_{l^j(b) l^j(\rho)}=\bigr(Y_{l^{k-1-j}(b) l^{k-1-j}(\rho)}\bigl)^T
 \end{equation}
Consequently, 
$$(-1)^j Y_{l^j(b)l^j(\rho)}+(-1)^{k-1-j} Y_{l^{k-1-j}(b)l^{k-1-j}(\rho)}=(-1)^j Y_{l^j(b)l^j(\rho)}+
(-1)^{k-1-j} \bigl(Y_{l^j(b)l^j(\rho)}\bigr)^T$$
and the latter matrix is symmetric if $k$ is odd and skew-symmetric if $k$ is even.
This together with \eqref{case12} implies the statement of the lemma.

{\it  Case 2.} Assume that $k> p$. 
We have $3$ subcases:
\begin{enumerate}
\item $k-p\leq j\leq p-1$;
\item $0\leq j<k-p$;
\item $p\leq j< k$.
\end{enumerate}
In subcase (1) \eqref{case11} holds, 
In subcase (2) $l^j(b)\in m(\Delta)$ and  $l^j(\rho)\in m(\Delta)$ and  in subcase (3) $l^j(b)\in \Delta$ and  $l^j(\rho)\in \Delta$.
In both of these subcases $\varepsilon\bigl(l^j(b)\bigr)=\varepsilon\bigl(l^j(\rho)\bigr)$. 
Hence, 
by \eqref{ddef}
\begin{equation}
\label{case21}
D(Y)_{b\rho}=\sum_{j=0}^{k-p-1}Y_{l^j(b)l^j(\rho)}+\sum_{j=k-p}^{p-1}(-1)^{j-k+p+1}Y_{l^j(b)l^j(\rho)}+(-1)^{2p-k}\sum_{j=p}^{k-1}Y_{l^j(b)l^j(\rho)}
\end{equation}
and by \eqref{rowinv}
\begin{equation}
\label{rowinv3}
Y_{l^j(b) l^j(\rho)}=-\bigr(Y_{l^{k-1-j}(b) l^{k-1-j}(\rho)}\bigl)^T
\end{equation}
The middle sum in \eqref{case21} corresponds to subcase (1) and can be treated as in the previous case.  To treat the other two sums note that by \eqref{rowinv3}
$$ Y_{l^j(b)l^j(\rho)}+(-1)^{k-2p} Y_{l^{k-1-j}(b)l^{k-1-j}(\rho)}=Y_{l^j(b)l^j(\rho)}+
(-1)^{k-1} \bigl(Y_{l^j(b)l^j(\rho)}\bigr)^T.$$
and the latter is symmetric if $k$ is odd and skew-symmetric if $k$ is even.
This implies that the sum of the first and the third sums in \eqref{case21} symmetric if $k$ is odd and skew-symmetric if $k$ is even, which completes the proof of the lemma.
\end{proof}
 
 \subsection{A class of $\mathrm{Ad}$-invariant normalization conditions}
Now we are ready to choose a normalization condition. Of course, this choice is not unique, but Proposition \ref{coboundimprop} and 
Lemma \ref{complementlem} immediately suggest an entire  class of normalization conditions invariant with respect to the adjoint action of the group $\mathrm {Sym}^{\mathrm{ne}}_{F(\delta)}$ (recall that it has the Lie algebra $\mathfrak u(\delta)=\mathfrak u_0(\delta)$).

Let us describe this class of normalization conditions. 
For any last box $\rho$  of a row of $\widetilde \Delta$ and a box $b\neq \rho$ , which is not higher than $\rho$, in the following set of pairs of boxes

\begin{equation}
\label{chainqb}
\{(b, \rho), \bigl(l(b), l(\rho)\bigr), \bigl(l^2(b), l^2(\rho)\bigr), \ldots\}
\end{equation}
we choose exactly one pair of boxes, denoted $\varphi(b,\rho)$.

Let $N_\varphi$ be the subspace of $\mathfrak g^0$ consisting of all $Y$ such that the only possibly nonzero matrices $Y_{ba}$ are one of the following:

\begin{enumerate}
	\item $Y_{\varphi(b,\rho)}$ , if $b$ and $\rho$ are  not in the same row of $\widetilde \Delta$;
	\item $Y_{\varphi(b,\rho)}$ and $Y_{m\bigl(\varphi(b,\rho)\bigr)}$, if  
	$b$ and $\rho$ are in the same row of  $\widetilde \Delta$  and $b\neq \rho$. Moreover, if $b$ is the $k$th box of the row, then $Y_{g\varphi(b,\rho)}$ is symmetric and $Y_{m\bigl(\varphi(b, \rho)\bigr)}= Y_{\varphi(b,\rho)}$, if $k$ is odd,  and
	$Y_{\varphi(b,\rho)}$ is skew-symmetric and $Y_{m\bigl(\varphi(b, \rho)\bigr)}= -Y_{\varphi(b,\rho)}$, if $k$ is even. 
\end{enumerate}	
 As a direct consequence of  Propositions \ref{prolongprop}, \ref{coboundimprop}, and Lemma \ref{complementlem}  we have the following
 
 \begin{thm}
 	\label{Ngthm}
 The subspace $N_\varphi$ corresponding to an assignment $\varphi$ as above is  an invariant  normalization condition  with respect to the adjoint action of the group $\mathrm {Sym}^{\mathrm{ne}}_{F(\delta)}$ on $\mathfrak g_{0}$.
 \end{thm}
 
 The invariancy follows from the form of the adjoint action given by \eqref{ADact}. The condition that $b\neq \rho$ comes from the fact that we need to find a complement  to $\mathfrak u(\delta)+[\delta, \mathfrak g^0]\cap\mathfrak g^0$ and not to $[\delta, \mathfrak g^0]\cap\mathfrak g^0$.
 
Note that  the normalization condition chosen in the original works \cite{li1, li2} belongs to the class of normalization of the previous theorem and it corresponds to the following assignment $\varphi_0$:
\begin{enumerate}
\item Assume that $b$ and $\rho$ are not in the same row of $\widetilde \Delta$ and assume that $c$ is the first (i.e. the most left) box of the row of $b$ in $\widetilde \Delta$  and $d$ is a box in the row of $\rho$ such that $(c,d)$ belongs to the set \eqref{chainqb}, then 
\begin{itemize}
	\item  if $m(c)$ is located to the left of $d$, then $\varphi_0(b, \rho)=(c,d)$;
	\item if $m(c)$ is not located to the left of $d$, then $\varphi_0(b, \rho)$ is the only pair in the set \eqref{chainqb} of the form $(m(b_1), a_1)$ or  $(m\circ r(b_1), a_1)$, where $a_1$ and $b_1$ lie in the same column;
\end{itemize}
\item  Assume that $b$ and $\rho$ are in the same row of $\widetilde \Delta$ and $b$ is the $k$th box of this row. 
\begin{itemize}
	\item If $k$ is odd then $\varphi_0(b, \rho)$ is the unique pair in the set \eqref{chainqb} of the form $(m(e), e)$;
	\item If $k$ is odd then $\varphi_0(b, \rho)$ is the unique pair in the set \eqref{chainqb} of the form $(m\circ r(e), e)$.
\end{itemize}
\end{enumerate} 
In the light of the more general theory developed here and based on \cite{flag2, flag1} this particular normalization $N_{\varphi_0}$ for general Young diagram does not have any advantage compared to any other normalization of Theorem \ref{Ngthm}.

Finally, assume  that the assignment $\varphi$ is chosen and we used Theorem \ref{normframe} to construct the bundle of moving frames.
Then for any $a\in \widetilde\Delta $ the space 
$$V_a(t):=\Gamma(t)(E_a),$$ is independent of the choice of the normal frame $\Gamma$  and it is endowed  with the canonical Euclidean structure (see Remark \ref{Orem}). Besides from the invariancy with respect to the adjoint action and \eqref{ADact} it follows that for any $a$ and $b$ the corresponding matrix block 
$\Bigl(C_\Gamma(t)\Bigr)_{ba}$ of the structure function of $\Gamma$ defines the linear map from $V_a$ to $V_b$ , which is independent of the choice of the normal basis. We  call it the \emph{$(a,b)$-curvature map of the curve $\gamma$} at $t$. \footnote{The $(a,b)$-curvature defined in the original work \cite{li1,li2} for the particular normalization condition chosen there corresponds to $\bigl(a, m(b)\bigr)$-curvature  here.} 

We conclude with several examples.

\begin{example}
	\label{regexamplec} This is the continuation of Example \ref {regexample}. The reduced diagram of a regular curve in Lagrangian Grassmannian consists of one box, say $\rho$. In this case there is only one box $b=m(\rho)$ in $\widetilde \Delta$, which differs from $\rho$ and hence, there is only one choice of the assignment $\varphi$, acting as the identity on the pair $\bigl(m(\rho), \rho\bigr)$. There is the unique nontrivial curvature map in this case, the $\bigl(m(\rho), \rho\bigr)$-curvature map, and it coincides with the curvature map of the regular curve defined in \cite{agrgam1, jac1}.
\end{example}

\begin{example} \emph{(The case of rectangular Young diagram)}
Assume that the Young diagram $D$ of $\Lambda$ is rectangular. Then the reduced  diagram $\Delta$ consists of only one row. Hence, for an assignment $\varphi$ the condition (1) for the corresponding normalization condition $N_\varphi$ is void. Now, if we use the assignment $\varphi_0$ as above, then for a normal frame $\Gamma$  the only possibly nonzero blocks  $\Bigl(C_\Gamma(t)\Bigr)_{ba}$ for its structure function  (where $b$ is not located to the right of $a$)  are when $(b,a)=(m(e), e)$ or $(b,a)=(m\circ r (e), e)$, where $e\in \Delta$. Moreover, the matrices $\Bigl(C_\Gamma(t)\Bigr)_{m(e)e}$ are symmetric and the matrices $\Bigl(C_\Gamma(t)\Bigr)_{m\circ r(e)e}$ are 
skew-symmetric.
\end{example}

\begin{example} (\emph{The case of rank $1$ curves in Lagrangian Grassmannians})
Let $\Lambda$ be an  equiregular and ample and  that \begin{equation}
\label{jump1}
\dim \Lambda^{(-1)}-\dim \Lambda =1.
\end{equation} The last condition is equivalent to the fact that the rank of the linear map $\frac{d}{dt} \Lambda(t)$ is equal to $1$. Such curves are called \emph{rank $1$ curves} in Lagrangian Grassmannians  and they appear as Jacobi curves of sub-Riemannian structures on rank $2$ distributions. From \eqref{jump1} and  the assumptions that the curve is ample and equiregular it follows  that 
$$\dim \Lambda^{(j-1)}-\dim \Lambda^{(j)}=1, \quad 0\leq  -j< \cfrac{1}{2}\dim V.$$	
Hence, the Young diagram $D$ of $\Lambda$ consists of one row of length $\cfrac{1}{2}\dim V$, i.e. this is a particular case of the previous example. 
In this case the corresponding matrices $\Bigl(C_\Gamma(t)\Bigr)_{ba}$ are $1\times 1$ matrix valued functions, i.e. they are usual (scalar-valued) functions.
and  if we use the normalization condition $N_{\varphi_0}$, then the only possibly nonzero entries  $\Bigl(C_\Gamma(t)\Bigr)_{ba}$ of the structure functions with $b$ located not to the right of $a$   are $\Bigl(C_\Gamma(t)\Bigr)_{m(e) e}$ with $e\in\Delta$, as a skew-symmmetric $1\times 1$ matrices are zero. 
Besides, by Remark \ref{Orem} the group $\mathrm{Sym}^{\mathrm{ne}}_{F(\delta)}$ is isomorphic to $\{\pm I\}$ and there are exactly two normal frames which differ by a sign. The tuple of functions  
\begin{equation}
\label{completeinv}
\left\{\Bigl(C_\Gamma(t)\Bigr)_{m(e) e}\right\}_{e\in\Delta}
\end{equation}
 for the complete system of invariants of the curve $\Lambda$, i.e. two curves are $Sp(V)$-equivalent if and only if the corresponding tuples as in \eqref{completeinv}
are equal.
\end{example}




{\bf Acknowledgenent}
\emph{I would like to thank David Sykes and Chengbo Li for editing and careful reading of this text. }

\end{document}